\documentclass[reqno,oneside,11pt]{amsart}
\usepackage[T1]{fontenc}
\usepackage[utf8]{inputenc}
\usepackage{lmodern}
\usepackage[french,english]{babel}
\usepackage{geometry} 

\usepackage{amsmath,amssymb,amsfonts,amsthm,mathtools}
\usepackage{mathrsfs}
\usepackage{graphicx}
\usepackage{caption}
\usepackage{subcaption} 
\usepackage{booktabs}
\usepackage{array}
\usepackage{paralist}
\usepackage{fancyhdr}
\usepackage{emptypage}
\fancypagestyle{plain}{
	\fancyhf{}
	\fancyfoot[RO,RE]{\thepage}
	
	}

\usepackage{comment} 
\usepackage{diagbox}
\usepackage{xcolor}
\usepackage{epigraph}

\usepackage[noadjust]{cite}
\usepackage[
pagebackref=true,
colorlinks=true,
urlcolor=purple,
linkcolor=purple!87!black,
citecolor=green!60!black,
pdfborder={0 0 0}
]{hyperref}
\renewcommand*{\backref}[1]{}
\renewcommand*{\backrefalt}[4]{[{\tiny%
		\ifcase #1 Not cited.%
		\or Cited on page~#2.%
		\else Cited on pages #2.%
		\fi%
	}]}
\usepackage{bookmark}
\usepackage{dsfont}
\newcommand{\id}{\mathsf{id}}

\newcommand{\Z}{\mathbb{Z}}

\renewcommand{\P}{\mathbb{P}}
\newcommand{\E}{\mathbb{E}}

\newcommand{\supp}[1]{\mathrm{supp}(#1)}

\newcommand{\std}{S_{\mathrm{std}}}

\usepackage{tikz}
\usetikzlibrary{patterns,positioning,arrows,decorations.markings,calc,decorations.pathmorphing,decorations.pathreplacing}
\usetikzlibrary{fadings}
\usepackage{forest}
\usepackage{tikz-qtree}
\RequirePackage{pgffor} 
\makeatletter
\newcommand{\neutralize}[1]{\expandafter\let\csname c@#1\endcsname\count@}
\makeatother

\newcommand{\thistheoremname}{}

\newtheorem*{genericthm*}{\thistheoremname}
\newenvironment{namedthm*}[1]
{\renewcommand{\thistheoremname}{#1}%
	\begin{genericthm*}}
	{\end{genericthm*}}
\theoremstyle{plain}
\newtheorem{thm}{Theorem}[section] 
\newtheorem{prop}[thm]{Proposition}
\newtheorem{lem}[thm]{Lemma}
\newtheorem{cor}[thm]{Corollary}
\theoremstyle{definition}
\newtheorem{rem}[thm]{Remark}
\newtheorem{defn}[thm]{Definition} 
\newtheorem{exmp}[thm]{Example} 
\newtheorem{question}[thm]{Question}

\allowdisplaybreaks
\address{Département de Mathématiques et Applications, \'{E}cole Normale Sup\'{e}rieure, PSL Research University, 45 rue d'Ulm, 75005, Paris, France.}
\email{eduardo.silva@ens.fr // edosilvamuller@gmail.com}
\newcommand{\Sym}[1]{\mathrm{FSym}(#1)}
\newcommand{\extsym}[1]{{\mathsf{Shuffler}}(#1)}

\usepackage{soul,xcolor}
\setstcolor{red}
\title[The Poisson boundary of lampshuffler groups]{The Poisson boundary of lampshuffler groups}
\date{\today}
\begin{document}
	\author{Eduardo Silva} 
	\maketitle
	\begin{abstract} 
We study random walks on the lampshuffler group $\Sym{H}\rtimes H$, where $H$ is a finitely generated group and $\Sym{H}$ is the group of finitary permutations of $H$. We show that for any step distribution $\mu$ with a finite first moment that induces a transient random walk on $H$, the permutation coordinate of the random walk almost surely stabilizes pointwise. Our main result states that for $H=\Z$, the above convergence completely describes the Poisson boundary of the random walk $(\Sym{\Z}\rtimes \Z,\mu)$.	\end{abstract}
\section{Introduction}\label{introduction}

We study random walks on the semi-direct product $\extsym{H}\coloneqq \Sym{H}\rtimes H$, where $H$ is an infinite countable group and $\Sym{H}$ denotes the group of bijections from $H$ to $H$ that coincide with the identity map outside of a finite set. Here, the action of an element $h\in H$ on a permutation $f\in \Sym{H}$ is defined as $(h\cdot f)(x)=hf(h^{-1}x)$, for $x\in H$. These groups are referred to as \emph{lampshuffler groups}\footnote{The name ``lampshuffler'' seems to have first appeared in \cite{GheysensMonod2022}, and it is used there to refer to the group $\mathrm{FAlt}(H)\rtimes H$, where $\mathrm{FAlt}(H)$ is the group of finitary even permutations of $H$. In this paper we use the name lampshuffler for the group $\Sym{H}\rtimes H$, following \cite{BonnetColinTesseraThomasse2022,TesseraGenevois2024}.} in \cite{GheysensMonod2022,BonnetColinTesseraThomasse2022, TesseraGenevois2024} due to their resemblance to lamplighter groups, and random walks on them are called mixer chains in \cite{Yadin2009}. In Section \ref{section: extensions of finitary symmetric groups} we describe the basic geometric and algebraic structure of the group $\extsym{H}$, and explain that it inherits the properties of $H$ being finitely generated, amenable, or elementary amenable (Lemma \ref{lem: basic properties of FSym(H)rtimes H}).

Let $G$ be a countable group and let $\mu$ be a probability measure on $G$. The \emph{(right) random walk} $(G,\mu)$ is the Markov chain with state space $G$ and with transition probabilities $p(g,h)=\mu(g^{-1}h)$, for $g,h\in G$. We assume that the random walk starts at the identity element $e_G\in G$. Random walks on lampshuffler groups have been studied in the literature. It is shown in \cite{Yadin2009} that the drift function of the simple random walk for the standard generating set on $\extsym{\Z}$ is asymptotically equivalent to $n^{3/4}$. In \cite[Corollary 1.4]{ErschlerZheng2020} it is proved that the F\o lner function of $\extsym{\Z^d}$, $d\ge 1$, is asymptotically equivalent to $n^{n^d}$, and the return probability $\mu^{2n}(e)$ of the simple random walk is shown to be asymptotically $\mathrm{exp}\left( -n^{\frac{d}{d+2}}\log^{\frac{2}{d+2}}n\right)$. Given a random walk $(G,\mu)$ and denoting by $h(\mu)$ its Avez asymptotic entropy (see Subsection \ref{subsection: conditional entropy} for the definition), the problem of ``full realization'' consists on realizing each number in the interval $[0, h(\mu)]$ as the Furstenberg entropy of some ergodic $(G, \mu)$-space. In \cite[Theorem 1.4]{HartmanYadin2018} it is proved that if $H$ is a finitely generated nilpotent group, then the lampshuffler group $\extsym{H}$ has full realization. Lampshuffler groups also appear in \cite{FeldheimSodin2020}, where the ``umpteen operator'' is introduced as a representation-theoretic analog of a random Schrödinger operator, and the property of having Lifshitz tails is linked to the decay of the return probability of the simple random walk on $\extsym{\Z^d}$.


The \emph{Poisson boundary} of a random walk $(G,\mu)$ is a measure space that encodes the asymptotic behavior of the process. It can be defined as the space of ergodic components of the shift map in the space of infinite trajectories. There are several other equivalent definitions of the Poisson boundary of a random walk, and we recall some of them in Section \ref{section: random walks}. In the last decades, there has been extensive research focused on the identification of Poisson boundaries, i.e., the problem of exhibiting an explicit measure space that coincides with the Poisson boundary up to a $G$-equivariant measurable isomorphism. The main result of the current paper is a complete description of the Poisson boundary of $\extsym{\Z}$, for measures $\mu$ with a finite first moment that induce a transient random walk on $\Z$.

Let $\mu$ be a probability measure on $\extsym{H}$ and consider the trajectory of the $\mu$-random walk $(F_n,S_n)\in \extsym{H}$, $n\ge 0$, where $F_n\in \Sym{H}$ is a finitely supported permutation of $H$ and $S_n\in H$. We refer to $\{F_n\}_{n\ge 1}$ as the \emph{permutation coordinate} of the $\mu$-random walk.

Our first result is the following \emph{stabilization lemma}.
\begin{lem}\label{lem: main stabilization lemma} Let $H$ be a finitely generated group, and consider a probability measure $\mu$ on $\extsym{H}$. Suppose that $\mu$ has a finite first moment and that it induces a transient random walk on $H$. Then for any $h\in H$, the values $F_n(h)$, $n\ge 0$, of the permutation coordinate of the random walk almost surely stabilize to a limit value $F_{\infty}(h)$.
\end{lem}
We prove this result in Section \ref{section: random walks on extsymH} in a more general form, where $H$ is not assumed to be finitely generated (Lemma \ref{lem. stabilization key lemma (Borel-Cantelli)}). If we furthermore suppose that $\mu$ is non-degenerate (i.e., that $\supp{\mu}$ generates $\extsym{H}$ as a semigroup), then the stabilization lemma shows that the Poisson boundary of $(\extsym{H},\mu)$ is non-trivial (Corollary \ref{cor: non degenerate + transient + finite expected support implies nontrivial poisson boundary}). In particular, the Poisson boundary of any simple random walk on $\extsym{H}$, for $H$ infinite and not virtually $\Z$ nor virtually $\Z^2$, is non-trivial. In contrast, simple random walks on $\extsym{\Z^d}$ for $d=1,2$ have a trivial Poisson boundary (see Section \ref{section: recurrent projection}). The well-known open ``stability problem'' asks whether the non-triviality of the Poisson boundary for a simple random walk on a finitely generated group depends on the choice of generating set. Corollary \ref{cor: non degenerate + transient + finite expected support implies nontrivial poisson boundary} together with Propositions \ref{prop: recurrent proj to Z has trivial Poisson boundary} and \ref{prop: recurrent proj to Z2 has trivial Poisson boundary} imply that for the family of groups $\extsym{\Z^d}$, $d\ge 1$, there is no such dependence.

It is a result of Rosenblatt \cite[Theorem 1.10]{Rosenblatt1981} and Kaimanovich and Vershik \cite[Theorem 4.4]{KaimanovcihVershik1983} that every amenable group admits a probability measure with a trivial Poisson boundary. Hence, if $H$ is infinite, amenable, and is not virtually $\Z$ nor virtually $\Z^2$, the group $\extsym{H}$ admits symmetric non-degenerate random walks with a transient projection to $H$, for which the permutation coordinate does not stabilize (Remark \ref{rem: not stabilization of FSym Z3 for nondegenerate rw}). In Proposition \ref{prop: existence of measure with trivial boundary but transient projection}, we prove that $\extsym{\Z}$ also admits random walks with this property. The stabilization lemma excludes measures where the permutation coordinate does not stabilize via the assumption of a finite first moment. In Proposition \ref{prop: a first moment condition is necessary for stabilization} we show that this condition cannot be weakened to the finiteness of a smaller moment: we construct a probability measure on $\extsym{\Z}$ that induces a transient random walk on $\Z$ and that has a finite $(1-\varepsilon)$-moment, for every $0<\varepsilon<1$, for which the permutation coordinate does not stabilize. 

We now state our main theorem.

\begin{thm}\label{thm. main theorem} Let $\mu$ be a probability measure on $\extsym{\Z}$ with a finite first moment that induces a transient random walk on $\Z$. Then the Poisson boundary of $(\extsym{\Z},\mu)$ coincides with the space of limit functions $F_{\infty}:\Z\to \Z$, endowed with the corresponding hitting measure.
\end{thm}
We prove this result by using Kaimanovich's Conditional Entropy Criterion \cite[Theorem 4.6]{Kaimanovich2000}. Another component of our proof is the \emph{displacement} associated with a permutation (Definition \ref{defn: displacement}). We explain the idea of the proof of Theorem \ref{thm. main theorem} at the beginning of Section \ref{section: proof of the main theorem}, and present the proof in Subsection \ref{subsection: the proof}. We mention that the conditional entropy criterion, together with the Ray criterion and Strip criterion that follow from it \cite{Kaimanovich2000}, has played a role in the identification of the Poisson boundary for many classes of groups, some of which we mention below.

The Poisson boundary has been described for several families of countable groups, many of which possess hyperbolic-like properties. One such family is that of non-abelian free groups, whose Poisson boundary was described in \cite{DynkinMaljutov1961} for measures supported on a free generating set and in \cite{Derrienic1975} for finitary measures. More generally, the Poisson boundary of a (non-elementary) hyperbolic group was shown to coincide with its Gromov boundary in \cite{Ancona1987} for finitary measures and in \cite[Theorems 7.4 and 7.7]{Kaimanovich2000} for more general $\mu$. Furthermore,  \cite{ChawlaForghaniFrischTiozzo2022} established that this description of the Poisson boundary holds for any measure of finite entropy on a hyperbolic group, and described the boundary for acylindrically hyperbolic groups, extending \cite[Theorem 1.5]{MaherTiozzo2021}. These papers cover the description of the Poisson boundary for various classes of groups that had already been studied, with extra conditions on the measure, such as groups with infinitely many ends \cite{Woess1989}, \cite[Theorem 8.4]{Kaimanovich2000}, mapping class groups \cite{KaimanovichMasur1996}, braid groups \cite{FarbMasur1998}, groups acting on $\mathbb{R}$-trees \cite{GauteroMatheus2012} and $\mathrm{Out}(F_n)$ \cite{Horbez2016}. Another family of groups we mention is that of discrete subgroups of semi-simple Lie groups, studied in \cite{Furstenberg1971,Ledrappier1985} and \cite[Theorems 10.3 and 10.7]{Kaimanovich2000}. 

The Poisson boundary has also been described for classes of groups that do not exhibit a hyperbolic nature. A notable family is that of amenable groups, which always admit a non-degenerate measure with a trivial Poisson boundary \cite{Rosenblatt1981,KaimanovcihVershik1983}. A natural question is whether \emph{every} non-degenerate measure on a given amenable group $G$ has a trivial Poisson boundary. In such a case, $G$ is called a \emph{Choquet-Deny group} and this family of groups includes abelian groups \cite{Blackwell1955, ChoquetDeny1960,DoobSnellWilliamson1960}, nilpotent groups \cite{DynkinMaljutov1961,Margulis1966}, and groups that have no ICC \footnote{A countable group is said to be \emph{ICC} if it is non-trivial and every non-trivial element has an infinite conjugacy class.} quotient \cite{Jaworski2004,LinZaidenberg1998}. It is proven in \cite{FrischHartmanTamuzVahidi2019} that the latter property is also necessary and thus provides an algebraic characterization of countable Choquet-Deny groups. This result is further developed in \cite{ErschlerKaimanovich2023}, where the authors prove the following: any countable group $G$ with an ICC quotient admits a non-degenerate symmetric measure of finite entropy, for which the Poisson boundary can be completely described in terms of the convergence of sample paths to the boundary of a locally finite forest, whose vertex set is $G$.

Kaimanovich and Vershik \cite[Proposition 6.4]{KaimanovcihVershik1983} provided the first examples of amenable groups that admit measures with a non-trivial Poisson boundary. Namely, they proved that for the wreath product $\Z/2\Z\wr \Z^d$, $d\ge 1$, and for any non-degenerate finitely supported measure $\mu$ whose projection to $\Z^d$ induces a transient random walk, the lamp configurations stabilize almost surely. This implies the non-triviality of the associated Poisson boundary, and it was conjectured that this space is completely described by the space of limit configurations. This was initially proven for measures with a projection to $\Z^d$ with non-zero drift \cite[Theorem 3.6.6]{Kaimanovich2001}, whereas the case of zero drift was proved first for $d\ge 5$ \cite{Erschler2011}, and later for $d\ge 3$ by Lyons and Peres \cite{LyonsPeres2021}. Their proofs can be adapted to provide an analogous description of the Poisson boundary of random walks on free metabelian groups \cite{Erschler2011,LyonsPeres2021} and, in both cases, generalize to infinitely supported measures with appropriate moment conditions. Additional results about the description of the Poisson boundary on wreath products have been obtained in \cite{KarlssonWoess2007,Sava2010} and we recall them in Subsection \ref{subsection: Poisson boundary of wreath products}. Another class of amenable groups for which the Poisson boundary has been completely described (for finitely supported measures) is the family of discrete affine groups of a regular tree, studied in \cite{BrieusselTanakaZheng}.

Theorem \ref{thm. main theorem} provides a new class of groups for which we have a complete description of the Poisson boundary. Since our result does not ask for non-degeneracy of $\mu$, and the group $\extsym{\Z}$ contains subgroups isomorphic to $F\wr \Z$ for every finite group $F$ (Proposition \ref{prop: extsymH contains wreath products whenever H is not co-Hopfian}), Theorem \ref{thm. main theorem} includes the description of the Poisson boundary for the wreath product of a finite group with an infinite cyclic base group.

In Subsection \ref{subsection: cyclic extensions of locally finite groups} we comment on the similarities and differences between the group $\extsym{\Z}$ and wreath products of the form $F\wr \Z$, where $F$ is finite or $F=\Z$. One distinction we remark is that $\extsym{\Z}$ admits (degenerate) measures with non-trivial Poisson boundary that induce a recurrent random walk on $\Z$, whereas this is not possible for wreath products $F\wr \Z$ as above. Nonetheless, in Section \ref{section: recurrent projection} we show that such examples cannot be found among finitely supported measures: we prove that any finitary measure on $\extsym{\Z}$ or on $\extsym{\Z^2}$ which induces a recurrent random walk on the base group has a trivial Poisson boundary (Propositions \ref{prop: recurrent proj to Z has trivial Poisson boundary} and \ref{prop: recurrent proj to Z2 has trivial Poisson boundary}).

There is also a difference between $\extsym{\Z}$ and wreath products in terms of the relation between the stabilization of configurations and the non-triviality of the Poisson boundary. In the case of $F\wr \Z$, for $F$ finite or $F=\Z$, and supposing that lamp configurations stabilize, the Poisson boundary will be non-trivial as soon as $\supp{\mu}$ contains two distinct elements with the same projection to $\Z$. Indeed, in such a case one can prove that there will be distinct limit configurations that occur with positive probability, which in turn provide non-trivial shift-invariant events in the space of infinite trajectories. In contrast, the analogous statement for $\extsym{\Z}$ does not hold. In Example \ref{example: self correcting configuration} we exhibit a family of virtually cyclic subgroups of $\extsym{\Z}$, generated by elements with the same projection to $\Z$, which has the following property: for every finitary measure with transient projection to $\Z$, the permutation coordinate of every trajectory stabilizes to the same limit function $F:\Z\to \Z$. Note that, since such subgroups are virtually cyclic, the Poisson boundary is trivial for any probability measure.
\subsection{Organization}
In Section \ref{section: extensions of finitary symmetric groups} we define the groups $\extsym{H}$, show that they contain wreath products as subgroups whenever $H$ is co-Hopfian, and comment on the similarities and differences between both families of groups. Afterward, in Section \ref{section: random walks} we recall preliminary facts about random walks and Poisson boundaries. In Section \ref{section: random walks on extsymH} we prove the Stabilization Lemma in a more general form (Lemma \ref{lem. stabilization key lemma (Borel-Cantelli)}). We prove Theorem \ref{thm. main theorem} in Section \ref{section: proof of the main theorem}. Finally, in Section \ref{section: recurrent projection} we show that finitary measures on $\extsym{H}$, for $H=\Z$ or $H=\Z^2$, whose projection to $H$ induces a recurrent random walk, have a trivial Poisson boundary.
\subsection{Acknowledgments} I would like to thank my advisor Anna Erschler for her guidance and help with this project, as well as Joshua Frisch for many stimulating discussions, and Bogdan Stankov for helpful stylistic revisions. I would like to thank the anonymous referee for their careful reading of the first version of this paper, and for their comments that improved the exposition. This project has received funding from the European Union’s Horizon 2020 research and innovation programme under the Marie Sk\l{}odowska-Curie grant agreement N\textsuperscript{\underline{o}} 945322, and from the European Research Council (ERC) under the European Union’s Horizon 2020 research and innovation program (grant agreement N\textsuperscript{\underline{o}} 725773).
\section{Extensions of finitary symmetric groups}\label{section: extensions of finitary symmetric groups}

Let $H$ be a countable group, and consider the group $\Sym{H}$ of bijective functions $f:H\to H$ such that only finitely many values $h\in H$ satisfy $f(h)\neq h$. We call the set of such elements the \emph{support} of $f$ and denote it by $\supp{f}$. Note that the group operation of $\Sym{H}$ is the composition of functions, in contrast to the direct sum $\bigoplus_{H}H$, where the group operation is pointwise multiplication. We will denote the identity element of $\Sym{H}$ by $\id$.

Since $H$ has its own group structure, there is a natural left action of $H$ on functions $f:H\to H$. More precisely, for every $h\in H$ and $f:H\to H$ we define $h\cdot f$ by
\begin{equation*}\label{eq: action of H on FSym(H)}
	(h\cdot f)(x)=hf(h^{-1}x), \text{ for }x\in H.
\end{equation*}

Whenever $f$ has finite support, so does $h\cdot f$ and one has $\supp{h\cdot f}=h\cdot \supp{f}$. Similarly, if $f$ is a bijection, then so is $h\cdot f$. With this, $H$ has a well-defined action on $\Sym{H}$, and we can consider the semi-direct product $\extsym{H}\coloneqq \Sym{H}\rtimes H$.

 We now mention some of the geometric properties of these groups that have been studied in the past decades, in addition to the ones related to random walks that were already discussed in the introduction. The lampshuffler group $\extsym{\Z}$ is considered in \cite[Section 2.3]{VershikGordon1997} as an example of a finitely generated group that is locally embeddable in the class of finite groups (LEF) that is not residually finite (in the same paper it is mentioned that this example goes back to Vershik's Doctor of Sciences thesis \cite{Vershik1973}). Additionally, the group $\extsym{\Z}$ was used in \cite{Stepin1983} as an example of a finitely generated non-residually finite group which admits a freely approximable action. This is in contrast with \cite[Theorem 1]{Stepin1983}, which states that any finitely presented group that admits such an action must be residually finite. The subgroup $\mathrm{FAlt}(\Z)\rtimes \Z$ of the lampshuffler group $\extsym{\Z}$, where $\mathrm{FAlt}(\Z)$ stands for the group of finitary even permutations of $\Z$, is used in \cite[Example 2]{Chou1980} as an example of the existence of free subsemigroups in elementary amenable groups that are not virtually solvable. In \cite[Theorem 3]{ElekSzabo2006} it is shown that if $H$ is an infinite, hyperbolic, residually finite group with Kazhdan's property (T), then the group $\extsym{H}$ is sofic but not residually amenable. The group $\extsym{\Z}$ is used in \cite[Proposition 4.4]{BrieusselZheng2019} to provide an example of a locally-finite-by-$\Z$ group that does not possess Shalom's property $H_{\mathrm{FD}}$.
\subsection{Basic properties} \label{subsection: generating set}
Note that the group $\Sym{H}$ is locally finite, meaning that every finitely generated subgroup is finite, and hence it is elementary amenable. Since (elementary) amenability is preserved by group extensions, the group $\extsym{H}$ is (elementary) amenable whenever $H$ is.

We now show that whenever $H$ is finitely generated then so is the group $\extsym{H}$, and exhibit an explicit standard choice of generators.

Define for any $x,y\in H$ the transposition $\delta^y_x\in \Sym{H}$  by \[\delta^y_x(h)=\begin{cases}
	y, &\text{ if }h=x, \\ x, &\text{ if }h=y, \text{ and}\\
	h &\text{ otherwise.}
\end{cases}\]
That is, $\delta_x^y$ corresponds to the bijection of $H$ that swaps $x$ and $y$, while leaving the rest of $H$ unchanged.
We denote $\delta_x\coloneqq \delta^{e_H}_x$, for $x\in H$. 

Suppose now that $H$ is finitely generated, and fix a finite generating set $S_H$ of $H$. Then a standard generating set for $\extsym{H}$ is given by $\std = S_H\cup \{\delta_s\mid s\in S_H \}.$ Indeed, it suffices to note that conjugating $\delta_s$ by generators of $S_H$ allows one to obtain transpositions between adjacent elements of any arbitrarily large ball of $H$, and thus any permutation supported on it.

Note that for any $(f,x)\in \extsym{H}$, we have $(f,x)\cdot (\id,h)=(f,xh),$ for $h\in H,$ as well as  $(f,x)\cdot (\delta_s,e_H)=(f\circ (x\cdot \delta_s),x)=(f\circ \delta^x_{xs}, x),$ for $s\in S_H.$ Hence, multiplying by elements of $H$ corresponds to a translation of the second coordinate, while multiplying by $\delta_s$ corresponds to precomposing the first coordinate by a transposition in the current position $x$ in the direction of $s$.

We summarize the above discussion in the following lemma.

\begin{lem}\label{lem: basic properties of FSym(H)rtimes H}
	Let $H$ be a countable group, and consider the extension $\extsym{H}\coloneqq \Sym{H}\rtimes H$.
	\begin{enumerate}
		\item If $H$ is finitely generated by $S_H$, then $\extsym{H}$ is also finitely generated and the set $\std\coloneqq S_H\cup \{\delta_s\mid s\in S_H\}$ is a finite generating set.
		
		\item If $H$ is amenable, then so is $\extsym{H}$.
		
		\item If $H$ is elementary amenable, then so is $\extsym{H}$.
	\end{enumerate}
\end{lem}

The following quantity will appear in the proof of Theorem \ref{thm. main theorem}.
\begin{defn}\label{defn: displacement}
	Given a word metric $d_H$ on $H$, we define the \emph{displacement} of a permutation $\sigma\in \Sym{H}$ by $		\mathrm{Disp}(\sigma)\coloneqq \sum_{h\in H}d_H(h,\sigma(h)).$
\end{defn}
Note that the above is well defined: since $\sigma$ is finitely supported, the values $d_H(h,\sigma(h))$ vanish for all but finitely many elements $h\in H$. We have the following lower bound for the word length in $\extsym{H}$ with respect to the generating set $\std$.

\begin{lem}\label{lem: word metric estimates} For every $(\sigma,x)\in \extsym{H}$,
	$$\|(\sigma,x) \|_{\std}\ge \max\left\{\frac{1}{2}\mathrm{Disp}(\sigma),\|x\|_H\right\}.$$
\end{lem}
\begin{proof}
	Since the only elements of $\std$ that have non-trivial projection to $H$ are those of $S_H$, it holds that $\|(\sigma,x) \|_{\std}\ge \|x\|_H$. Indeed, the above implies that any geodesic word of length $n$ for $(\sigma,x)$ in $\std$ projects to a word of length $n$ in $S_H$ that evaluates to $x$ in $H$.
	
	On the other hand, each multiplication by a transposition $\delta_s\in \std$ changes the value of $\mathrm{Disp}(\cdot)$ by at most $2$ units. This implies that
	$\mathrm{Disp}(\sigma)\le 2\|(\sigma,x)\|_{\std}$.
\end{proof}
\subsection{Wreath products subgroups in lampshufflers} \label{subsection: lamplighter subgroups}
Given groups $A$ and $B$, we recall that their wreath product is defined by $A\wr B\coloneqq \bigoplus_{B}A\rtimes B$ (see also Subsection \ref{subsection: Poisson boundary of wreath products}). As we mentioned in the introduction, it is natural to compare $\extsym{H}$ to wreath products of the form $F\wr H$, for $F$ a finite non-trivial group. The objective of this subsection is to show a condition that guarantees that $\extsym{H}$ contains a subgroup isomorphic to a wreath product $F\wr H$ with $F$ an arbitrary finite group. In particular, we will see that this holds for $H=\Z^d$, $d\ge 1$, as well as for any free group.

Recall that a group $H$ is called \emph{co-Hopfian} if every injective endomorphism of $H$ is an isomorphism. In other words, a group is co-Hopfian if and only if it is not isomorphic to a proper subgroup of itself.
\begin{prop}\label{prop: extsymH contains wreath products whenever H is not co-Hopfian} Let $H$ be an infinite group that is not co-Hopfian. Then for every finite group $F$, the group $\extsym{H}$ contains a subgroup isomorphic to $F\wr H$.
\end{prop}
\begin{proof}
Suppose that $K\leqslant H$ is a proper subgroup with $K\cong H$ and fix $n\ge 1$. Then we can choose $K$ such that its index $[H:K]$ is some value $m\in \{n,n+1,\ldots\}\cup \{\infty\}$. Indeed, if $\varphi:H\to H$ is an injective homomorphism then $[H:\varphi^n(H)]=[H:\varphi(H)]^n$. The subgroup generated by $K$ together with the (finite) symmetric group on a section of $H/K$ generates $\mathrm{FSym}(m)\wr K\cong \mathrm{FSym}(m)\wr H$, where $\mathrm{FSym}(m)$ is the group of finitary permutations on a set of cardinality $m$. Since any finite group embeds into $\mathrm{FSym}(m)$ for some $m\ge 2$, this shows the proposition.
\end{proof}

Examples of groups that are not co-Hopfian are free abelian groups, free groups, solvable Baumslag-Solitar groups \cite[Section III.22]{delaHarpe2000} and Thompson's group $F$ \cite{Wassink2011}. In contrast, the family of co-Hopfian groups includes lattices in semi-simple Lie groups \cite{Prasad1976}, one-ended torsion-free hyperbolic groups \cite{Sela1997}, the group $\mathrm{Out}(F_n)$, $n\ge 3$ \cite{FarbHandel2007,HorbezWade2020}, and the mapping class group of a closed hyperbolic surface \cite{IvanovMcCarthy1999}.

In the case of $H=\Z^d$, $d\ge 1$, we can show that the embedding of a wreath product into $\extsym{\Z^d}$ does not distort its word metric. Recall that given two metric spaces $(X,d_X)$ and $(Y,d_Y)$, the space $(X,d_X)$ \emph{embeds quasi-isometrically} into $(Y,d_Y)$ if there exists a function $f:X\to Y$ and constants $C\ge 0$, $K\ge 1$ such that for every $x_1,x_2\in X$,
\begin{equation*}
	\frac{1}{K}d_X(x_1,x_2)-C\le d_Y(f(x_1),f(x_2))\le Kd_X(x_1,x_2)+C.
\end{equation*}

\begin{prop} For any $d\ge 1$ and any (non-trivial) finite group $F$, the wreath product $F\wr \Z^d$ embeds quasi-isometrically into $\extsym{\Z^d}$.
\end{prop}
\begin{proof}

Let $F$ be any non-trivial and finite group. Consider $m\ge 2$ such that $F$ is isomorphic to a subgroup of $\mathrm{Sym}(m)$ the symmetric group on $m$ elements. Since $F\wr \Z^d$ embeds quasi-isometrically into $\mathrm{Sym}(m)\wr \Z^d$, in order to prove the proposition it suffices to show that $\mathrm{Sym}(m)\wr \Z^d$ embeds quasi-isometrically into $\extsym{\Z^d}$. 

Throughout the proof we will consider $\Z^d$ and its finite index subgroup $H\coloneqq (m\Z)^d$. In what follows we will use additive notation when referring to the group operation of $\Z^d$ or $H$.

 Let us denote $T\coloneqq\{0,1,\ldots,m-1\}^d\subseteq \Z^d$, which corresponds to a set of representatives for $\Z^d/H$. Then $\Z^d$ is partitioned into the cosets $\mathbf{h}+T$, for $\mathbf{h}\in H$. This implies that we have the embedding of $\mathrm{Sym}(m)\wr \Z^d\cong \mathrm{Sym}(T)\wr H\leqslant\extsym{\Z^d}$. In order to show that this is a quasi-isometric embedding, let us introduce finite generating sets for these groups.

 Consider $S_{\Z^d}=\{\pm\hat{\mathbf{e}}_1,\pm\hat{\mathbf{e}}_2,\ldots, \pm\hat{\mathbf{e}}_d,\mathbf{0}\}$ the canonical generating set of $\Z^d$ together with the identity element $\mathbf{0}\in \Z^d$. Consider the generating set $S_{H}=\{\pm m \hat{\mathbf{e}}_1,\pm m\hat{\mathbf{e}}_2,\ldots, \pm m\hat{\mathbf{e}}_d,\mathbf{0}\}$ for $H$. With this, we define the finite generating sets
 
 	$$S_{\mathrm{wreath}}= \left\{(\sigma, s)\mid \mathrm{Sym}(T) \text{ and }s\in  S_H \right\}$$
 	 for $\mathrm{Sym}(T)\wr H$, and
 	  $$S_{\mathrm{ext}}= \left\{(\sigma, s)\mid \mathrm{Sym}(T) \text{ and }s\in  S_{\Z^d}\cup S_{H}\right\}$$ 
 	   
 	 for $\extsym{\Z^d}.$ Note that $S_{\mathrm{wreath}}\subseteq S_{\mathrm{ext}}$, and hence that $\|g\|_{S_{\mathrm{ext}}}\le \|g\|_{S_\mathrm{wreath}}$ for all $g\in\mathrm{Sym}(T)\wr H $.
 	  In order to finish the proof, it suffices to show the existence of a constant $C>0$ such that for every $g\in\mathrm{Sym}(T)\wr H $, we have \begin{equation}\label{eq: geodesics}
 	  	\|g\|_{S_{\mathrm{wreath}}}\le C \|g\|_{S_{\mathrm{ext}}}.
 	  \end{equation} 
 	 
 	 Let us consider an arbitrary element $g=(f,x)\in\mathrm{Sym}(T)\wr H $, and denote $n\coloneqq \|g\|_{S_{\mathrm{ext}}}$. 
 	   
 	 Consider a word $w_1w_2\cdots w_n$ with $w_i\in S_{\mathrm{ext}}$ for each $i=1,\ldots, n$ and $g=w_1w_2\cdots w_n$. For every $i=1,\ldots, n$, define $\mathbf{h}_i\in H$ to be the unique element such that the projection to $\Z^d$ of $w_1\cdots w_i$ belongs to $\mathbf{h}_i+T$. Also denote $\mathbf{h}_0\coloneqq\mathbf{0}\in \Z^d$. In particular, we have $\mathbf{h}_n=x$, which corresponds to the projection of $g$ to $\Z^d$. We remark additionally that for every $i=0,1,\ldots, n-1$, we have that $h_{i+1}=h_i + \sum_{j=1}^d \varepsilon_j m  \hat{\mathbf{e}}_j$ for some values $\varepsilon_j\in \{-1,0,1\}$, for $j=1,\ldots, d$.
 	 
 	 Note that for every $\mathbf{h} \in H$, the permutation $f$ restricts to a bijection $f:\mathbf{h}+T\to \mathbf{h}+T$, and that 
 	 $$\supp{f}\subseteq \bigcup_{i=0}^n(\mathbf{h}_i+B+T),$$
 	 where $B=\left\{ \sum_{i=1}^d \varepsilon_i m \hat{\mathbf{e}}_i\mid \varepsilon_1,\ldots,\varepsilon_d\in \{0,1\} \right\}$. This is since the function $f$ can only act non-trivially on copies of $T$ that correspond to $\mathbf{h}_i+T$, $i=0,1,\ldots, n$, or that are neighboring to one of these cosets.
 	 
 	 Then, since the generating set $S_{\mathrm{wreath}}$ allows to have any permutation in $\mathrm{Sym}(T)$ accompanying any generator $s\in S_H$, in order to show Equation \eqref{eq: geodesics} it suffices to prove the following statement: there is a constant $C>0$, that only depends on $d$, such that there is a path of length at most $Cn$ on the Cayley graph of $H$ with respect to $S_H$, that starts at $\mathbf{0}\in H$, finishes at $x\in H$ and which visits all elements in the set \[ \{\mathbf{h}_i+b\mid i=0,1,\ldots, n \text{ and }b\in B\}. \]
 	 
 	 For every $i=0,1,\ldots,n-1$ we consider a path $p_i$ on the Cayley graph of $H$ with respect to $S_H$ (whose vertices are identified with the lattice $(m\Z)^d\subseteq \Z^d$) that begins at $\mathbf{h}_i$, finishes at $\mathbf{h}_{i+1}$ and visits all elements in the set $\{\mathbf{h}_i+b\mid b\in B\}$. The length of such a path can be chosen to be less than or equal to $2d\cdot 2^d+d$. Indeed, the path $p_i$ can be formed as follows: consider the cycles that go from $\mathbf{h}_i$ to $\mathbf{h}_i+b$ and back to $\mathbf{h}_i$ for each  $b\in B$, which each has length at most $2d$, and there are $2^d$ elements in $B$. By concatenating these cycles together with a path from $\mathbf{h}_i$ to $\mathbf{h}_{i+1}$, which can be chosen to have length at most $d$, we obtain the path $p_i$ of length at most $2d\cdot 2^d+d$. Finally, we concatenate all paths $p_i$, $i=0,1,\ldots, n$, and obtain a path of length at most $(2d\cdot 2^d+d)n$ on the Cayley graph of $H$ with respect to $S_H$, that starts at $\mathbf{0}\in H$, finishes at $x\in H$ and which visits all elements in the set \[ \{\mathbf{h}_i+b\mid i=0,1,\ldots, n \text{ and }b\in B\}. \] 
 	 
 	 This path can then be used, as explained two paragraphs above, to construct a word of length at most $(2d\cdot 2^d+d)n$ using generators from $S_{\mathrm{wreath}}$ that evaluates as a group element to $g$. This finishes the proof, since we have proved that one can choose $C=2d\cdot 2^d+d$ so that Equation \eqref{eq: geodesics} holds.
\end{proof}

\subsection{Cyclic extensions of locally finite groups}\label{subsection: cyclic extensions of locally finite groups}

Our main theorem (Theorem \ref{thm. main theorem}) concerns the group $\extsym{\Z}$, which is an extension by $\Z$ of the locally finite group $\Sym{\Z}$. In this subsection, we discuss the geometric properties present in groups that exhibit this algebraic structure and motivate the study of random walks on them.

Recall that a group $\mathcal{L}$ is called \emph{locally finite} if every finitely generated subgroup is finite. A group $G$ is a \emph{cyclic extension of a locally finite group} or  \emph{locally-finite-by-$\Z$} if there is a short exact sequence
\begin{equation}\label{eq: short exact sequence def locally finite by Z}
	1\rightarrow \mathcal{L}\rightarrow G\rightarrow \Z\rightarrow 1,	
\end{equation}
where $\mathcal{L}$ is a locally finite group. Observe that any such sequence must necessarily split, so that $G=\mathcal{L}\rtimes \Z$. In particular, the group $\extsym{\Z}$ as well as wreath products of the form $F\wr \Z=\bigoplus_{\Z}F\rtimes \Z$, for a non-trivial and finite group $F$, are cyclic extensions of locally finite groups. Other groups in this category are the discrete affine groups of a regular tree. These groups were introduced in \cite{BrieusselTanakaZheng}, where their Poisson boundary was described for finitary measures. If every finitely generated subgroup of $\mathcal{L}$ is contained in a finite \emph{normal} subgroup of $\mathcal{L}$, then we say that $\mathcal{L}$ is \emph{locally normally finite}. Notably, $\bigoplus_{\Z}F$ is locally normally finite, while $\Sym{\Z}$ is not. This provides an algebraic distinction between the groups $F\wr \Z$ and $\extsym{\Z}$. Cyclic extensions of locally normally finite groups are studied in \cite{BrieusselZheng2019}.

Cyclic extensions of locally finite groups can have arbitrarily fast-growing F\o lner functions. This was first observed in the last remark of Section 3 in \cite{Erschler2003} (see also \cite[Section 8.2]{Gromov2008}), and a proof of this result is given in \cite[Corollary 1.5]{OlshanskiiOsin2013}. In \cite[Theorem 1.1]{BrieusselZheng2021}, it is proven that locally-finite-by-$\Z$ groups exhibit a large class of speed, return probability, entropy, isoperimetric profiles, and $L_p$-compression functions. In particular, the authors prove that any sufficiently regular function that grows at least exponentially can be realized as the F\o lner function of a locally-finite-by-$\Z$ group, which admits a simple random walk with a trivial Poisson boundary \cite[Corollary 4.7]{BrieusselZheng2021}. This behavior differs from what happens for the \textit{linear algebraic F\o lner function}, introduced in \cite[Section 1.9]{Gromov2008}. In general, this function is bounded above by the usual (combinatorial) F\o lner function, and both of them coincide for left-orderable groups \cite[Section 3.2]{Gromov2008}. In contrast, the linear algebraic F\o lner function of every locally-finite-by-$\Z$ group grows linearly \cite[Section 8.1]{Gromov2008}.

There are locally-finite-by-$\Z$ groups that admit simple random walks with a non-trivial Poisson boundary. Examples of this are $\Z/2\Z\wr \left( \Z/2\Z\wr \Z \right)$ and the discrete affine group of a regular tree \cite{BrieusselTanakaZheng}. On the other hand, we mentioned above that \cite[Corollary 4.7]{BrieusselZheng2021} provides a large family of locally-finite-by-$\Z$ groups that admit simple random walks with a trivial Poisson boundary. As a more concrete example, simple random walks on the groups $F\wr \Z$, for $F$ finite or $F=\Z$, and $\extsym{\Z}$ also have a trivial boundary. The case of wreath products is proven in \cite[Proposition 6.2]{KaimanovcihVershik1983}, while the case of $\extsym{\Z}$ follows from the sublinear asymptotics of the drift function \cite{Yadin2009}. Moreover, every probability measure on $F\wr \Z$ with a recurrent projection to the base group $\Z$ has a trivial Poisson boundary. This is proven in \cite[Proposition 6.3]{KaimanovcihVershik1983} and \cite{Kaimanovich1983} for finitary measures, in \cite{Kaimanovich1991} for finite first moment measures, and in \cite[Proposition 4.9]{LyonsPeres2021} for the general case, by using the classification of Choquet-Deny groups \cite{FrischHartmanTamuzVahidi2019}. In contrast, this generalization does not hold for $\extsym{\Z}$. Indeed, the group $\Sym{\Z}$ admits symmetric measures with finite entropy and with a non-trivial Poisson boundary \cite{Kaimanovich1983} (this also follows from \cite{FrischHartmanTamuzVahidi2019}, since every non-trivial element of $\Sym{\Z}$ has an infinite conjugacy class). In consequence, $\extsym{\Z}$ admits (degenerate) random walks with a non-trivial Poisson boundary and a recurrent projection to $\Z$.
\section{Background on random walks and Poisson boundaries}\label{section: random walks}
\subsection{Random walks on groups}
Let $G$ be a countable group and $\mu$ a probability measure on $G$. The \emph{(right) $\mu$-random walk} $(G,\mu)$ is the Markov chain with state space $G$, whose transition probabilities are given by 
\begin{equation*}
	p(g,h)\coloneqq \mu(g^{-1}h), \ g,h\in G.
\end{equation*}

We assume that the $\mu$-random walk starts at the identity $e_G\in G$. The space of infinite trajectories of the random walk $G^{\infty}$ is endowed with the probability $\P$, which is the push-forward of the Bernoulli measure $\mu^{\mathbb{N}}$ on the space of increments $G^{\infty}$ through the map
\begin{align*}
	G^{\infty}&\to G^{\infty} \\
	(g_1,g_2,g_3,\ldots)&\mapsto (g_1,g_1g_2,g_1g_2g_3,\ldots).
\end{align*}

Suppose that $G$ is finitely generated, and let $\ell_G$ be a word length on $G$. For $\alpha>0$, a probability measure $\mu$ on $G$ is said to have a \emph{finite $\alpha$-moment} if $\sum_{g\in G}\ell_G(g)^{\alpha}\mu(g)<\infty$. This property does not depend on the choice of $\ell_G$, since changing the word length on a group modifies the metric by a multiplicative constant.

\subsection{The Poisson boundary}\label{subsection: the Poisson boundary}
Let us now recall the definition of the Poisson boundary, together with some equivalent definitions. For background on random walks on countable groups and their Poisson boundaries, we refer to \cite{KaimanovcihVershik1983,Kaimanovich2000}, the surveys \cite{Erschler2010,Zheng2022}, and the introduction and Section 3 of \cite{ErschlerKaimanovich2023}.

Given two trajectories $\mathbf{x}=(x_1,x_2,\ldots)$ and $\mathbf{y}=(y_1,y_2,\ldots)$ in $G^{\infty}$, say that they are \emph{orbit equivalent} if there exist $p,N\ge 0$ such that $x_{p+n}=y_{n}$ for every $n\ge N$. Consider the measurable hull of this equivalence relation in $G^{\infty}$. That is, the $\sigma$-algebra of measurable subsets of $G^{\infty}$ which are unions of the equivalence classes, modulo $\P$-null sets. The associated quotient of the space of infinite trajectories by this measurable hull is called the \emph{Poisson boundary} $\partial_{\mu} G$ of the random walk $(G,\mu)$. Equivalently, the Poisson boundary is the space of ergodic components of the \emph{shift map} of the space of infinite trajectories, defined by
\begin{equation*}
	\begin{aligned}
		T:G^{\infty}&\to G^{\infty}\\
		(x_1,x_2,x_3,\ldots)&\mapsto (x_2,x_3,\ldots).
	\end{aligned}
\end{equation*}

If we do not allow the shift by $p$ in the definition of the orbit equivalence relation above, we obtain the \emph{tail equivalence} relation. The associated quotient space is called the \emph{tail boundary} of the random walk, and it provides an alternative and equivalent definition of the Poisson boundary of a random walk on a group \cite{Derrienic1980,KaimanovcihVershik1983}. We remark that the Poisson boundary can be defined for general Markov chains, and these two definitions are no longer equivalent in this broader context (see \cite[Example 2]{BlackwellFreedman1964} and \cite[Theorem 2.2]{Kaimanovich1992}).

Recall that a function $f:G\to \mathbb{R}$ is called \emph{$\mu$-harmonic} if for every $g\in G$, it holds that $f(g)=\sum_{h\in G}f(gh)\mu(h)$. The Poisson boundary of $(G,\mu)$ can be described in terms of the space of bounded $\mu$-harmonic functions on the subgroup generated by $\supp{\mu}$. In particular, the non-triviality of the Poisson boundary of a non-degenerate random walk is equivalent to the existence of non-constant bounded $\mu$-harmonic functions on $G$. 

Now we introduce the concept of a $\mu$-boundary of $G$. Let us denote by $\mathbf{bnd}:G^{\infty}\to \partial_{\mu} G$ the associated quotient map from the space of trajectories onto the Poisson boundary. The space $\partial_{\mu}G$ is endowed with the so-called \emph{harmonic measure} $\nu\coloneqq \mathbf{bnd}_{*}(\P)$, which satisfies the equation $\mu*\nu=\nu$. One says that $\nu$ is \emph{$\mu$-stationary}. Thus, the Poisson boundary of $(G,\mu)$ has the structure of a measure space $(\partial_{\mu} G,\mathcal{F},\nu)$ endowed with a measurable $G$-equivariant map $\mathbf{bnd}:G^{\infty}\to \partial_{\mu} G$ that satisfies

\begin{enumerate}
	\item $\mathcal{I}=\mathbf{bnd}^{-1}(\mathcal{F})$ modulo $\P$-null sets, where $\mathcal{I}$ is the sub-$\sigma$-algebra of shift-invariant events of the space of trajectories $G^{\infty}$, and
	
	\item the measure $\nu=\mathbf{bnd}_{*}(\P)$ is $\mu$-stationary, which by definition means that it satisfies the equation $\nu=\mu*\nu\coloneqq \sum_{g\in G}\mu(g)g\nu.$
\end{enumerate}

In general, a measure space $(B,\mathcal{A},\lambda)$ endowed with a measurable $G$-action is called a \emph{$\mu$-boundary} of $G$ if there exists a $G$-equivariant measurable map $\pi:G^{\infty}\to B$ for which $\lambda=\pi_{*}(\P)$ is $\mu$-stationary and such that $\pi^{-1}(\mathcal{A})\subseteq \mathcal{I}$ modulo $\P$-null sets. The Poisson boundary of $(G,\mu)$ is the \emph{maximal} $\mu$-boundary of $G$, in the sense that for every $\mu$-boundary $(B,\mathcal{A},\lambda)$, the projection $\pi:G^{\infty}\to B$ factors through the map $\mathbf{bnd}$, and it is unique up to a $G$-equivariant measurable isomorphism.
\subsection{The Poisson boundary of wreath products}\label{subsection: Poisson boundary of wreath products}
We illustrate the above notions with the case of wreath products, which also serves as a point of comparison with the results of this paper.

Recall that given groups $A$ and $B$, their wreath product $A\wr B$ as the semidirect product $\bigoplus_{B}A\rtimes B$, where $B$ acts by translations on the direct sum $\bigoplus_{B}A$. We remark that some authors use the notation $B\wr A$. Wreath products are also called ``lamplighter groups'', and for an element $(f,x)\in A\wr B$, the function $f$ is referred to as the ``lamp configuration''.

Kaimanovich and Vershik \cite{KaimanovcihVershik1983} proved that for $\Z/2\Z\wr \Z^d$, $d\ge 1$, and for any non-degenerate finitely supported measure $\mu$ whose projection to $\Z^d$ induces a transient random walk, the lamp configurations stabilize almost surely. In other words, with probability one, the values assigned by the lamp configuration to a given finite subset of $\Z^d$ change only finitely many times along the trajectory of the random walk. Hence, the space of limit lamp configurations has the structure of a $\mu$-boundary, and its non-triviality implies that the Poisson boundary is non-trivial as well. These were the first examples of measures with non-trivial boundary on amenable groups, and Kaimanovich and Vershik conjectured that the space of limit configurations coincides with the Poisson boundary. This was proven under the hypotheses of a finite first moment of $\mu$ and projection to $\Z^d$ with non-zero drift by Kaimanovich \cite[Theorem 3.6.6]{Kaimanovich2001}, whereas the case of zero drift, for $d\ge 3$, remained open. The Poisson boundary was later described for $A\wr B$ where $A$ is finite and $B$ is a free group \cite{KarlssonWoess2007}, or where $B$ has infinitely many ends, or is hyperbolic \cite{Sava2010}. The case of measures on $A\wr \Z^d$ with a centered projection to $\Z^d$ was proved for $d\ge 5$ by Erschler \cite{Erschler2011}, and by Lyons and Peres for $d\ge 3$ \cite{LyonsPeres2021}, both with additional moment conditions on $\mu$.
\subsection{The conditional entropy criterion}\label{subsection: conditional entropy}
Recall that the \emph{entropy} $H(\mu)$ of a probability measure $\mu$ on $G$ is defined as
$H(\mu)\coloneqq -\sum_{g\in G}\mu(g)\log(\mu(g))$. Avez \cite{Avez1972} introduced the \emph{asymptotic entropy} of the random walk $(G,\mu)$, defined as $h(\mu)\coloneqq \lim_{n\to \infty}H(\mu^{*n})/n$. The existence of this limit is guaranteed by the subadditivity of the sequence $H_n\coloneqq H(\mu^{*n})$, $n\ge 1$. Avez \cite{Avez1974} proved that if $h(\mu)=0$, then the random walk has a trivial Poisson boundary. Furthermore, the \emph{Entropy Criterion}, due to Derrienic \cite{Derrienic1980} and Kaimanovich and Vershik \cite{KaimanovcihVershik1983} states that if $H(\mu)<\infty$, then $h(\mu)=0$ if and only if the Poisson boundary of the $\mu$-random walk on $G$ is trivial. This criterion was strengthened by Kaimanovich to a \emph{Conditional Entropy Criterion}, which we state below. 

\begin{thm}[{\cite[Theorem 4.6]{Kaimanovich2000}}]\label{thm: entropy criterion formulation} Let $\mu$ be a probability measure on $G$ with finite entropy, and consider $\mathbf{B}=(B,\mathcal{A},\nu)$ a $\mu$-boundary of $G$. Suppose that for every $\varepsilon>0$ there exists a random sequence of finite subsets $\{Q_{n,\varepsilon}\}_{n\ge 1}$ of $G$ such that
	\begin{enumerate}
		\item \label{item: strip criterion 1} the random set $Q_{n,\varepsilon}$ is a measurable function with respect to $\mathcal{A}$, for every $n\ge 1$;
		\item \label{item: strip criterion 2} $\displaystyle\limsup_{n\to \infty}\frac{1}{n}\log|Q_{n,\varepsilon}|<\varepsilon$ almost surely; and
		\item \label{item: strip criterion 3} $\displaystyle\limsup_{n\to \infty} \P\left(x_n\in Q_{n,\varepsilon} \right)>0$, where $\{x_n\}_{n\ge 1}$ is the trajectory of the $\mu$-random walk.
	\end{enumerate}
	Then $\mathbf{B}$ coincides with the Poisson boundary of $(G,\mu)$.
\end{thm}
This result is a common tool used to prove the maximality of a $\mu$-boundary for a random walk on a group, and we will use it in the proof of Theorem \ref{thm. main theorem}. We mention that Lyons and Peres \cite[Corollary 2.3]{LyonsPeres2021} proved an alternative version of this criterion, where the third condition of Proposition \ref{thm: entropy criterion formulation} is replaced by
\begin{equation*}
	\limsup_{n\to \infty} \P\left(\text{ there exists }m\ge n \text{ such that }x_m\in Q_{n,\varepsilon} \right)>0.
\end{equation*}
This condition is easier to verify in some situations \cite{LyonsPeres2021,BrieusselTanakaZheng}. For Theorem \ref{thm. main theorem} we will apply the original version.

\section{Random walks on $\extsym{H}$}\label{section: random walks on extsymH}
\subsection{Stabilization of the permutation coordinate}
We study conditions on the measure $\mu$ that guarantee the stabilization of the permutation coordinate to a limit function $F_{\infty}:H\to H$. We first give a precise definition of stabilization.

\begin{defn}\label{def: stabilization of permutation coordinate} We say that the permutation coordinate $\{F_n\}_{n\ge 1}$ of the $\mu$-random walk $\left\{(F_n,S_n)\right \}_{n\ge 0}$ on $\extsym{H}$ \emph{stabilizes} if almost surely for every $h\in H$, there exists $N\ge 1$ such that $F_n(h)=F_N(h)$ for all $n\ge N$.
\end{defn}

Whenever the permutation coordinate stabilizes, we can associate with almost every trajectory $\left\{(F_n,S_n)\right \}_{n\ge 0}$ of the $\mu$-random walk a limit function $F_{\infty}:H\to H$, which satisfies for every $h\in H$, $F_n(h)=F_{\infty}(h)$ for large enough $n$. Since every function $F_n$ is a bijection from $H$ to itself, $F_{\infty}$ will be injective. However, it may happen that the limit function is not surjective.

\begin{exmp}\label{rem: limiting config is not always surjective}
	Consider $F(X)$ the free group on a finite set $X\neq \varnothing$, and let $\mu$ be any probability measure on $\extsym{F(X)}$ whose support is the set $\{(\delta_x,x)\mid x\in X\}$. The projection $\{S_n\}_n$ of the $\mu$-random walk to $F(X)$ is supported on the free semigroup generated by $X$, and hence it will converge to a geodesic ray $\gamma:[0,+\infty) \to F(X)$. Since the support of $\mu$ consists of elements of the form $(\delta_x,x)$, the permutation coordinate of the $\mu$-random walk will stabilize to a limit function $F_{\infty}$, that satisfies $F_{\infty}(\gamma(i))=\gamma(i+1)$, $i\ge 0$, and $F_{\infty}(h)=h$ for any $h$ outside of the image of $\gamma$. In particular, the identity element does not have a preimage through $F_{\infty}.$ Note also that the Poisson boundary of $(\extsym{F(X)},\mu)$ is non-trivial whenever $|X|\ge 2$.
\end{exmp}

\begin{prop}\label{prop: non-degenerate measure + stabilization imply non-trivial boundary} Let $H$ be an infinite countable group and $\mu$ a non-degenerate probability measure on $\extsym{H}$. Suppose that the permutation coordinate of the $\mu$-random walk stabilizes. Then the Poisson boundary of $(\extsym{H},\mu)$ is non-trivial.
\end{prop}
\begin{proof}
	Denote by $F_{\infty}:H\to H$ the limit function of the permutation coordinate of the random walk $(F_n,S_n)$ on $\extsym{H}$. If the Poisson boundary is trivial, then the function $F_{\infty}$ is the same for almost every trajectory of the random walk. Consider the left action of $\extsym{H}$ on the space of limit functions $F_{\infty}$. The stabilization together with the non-degeneracy assumption imply that $F_{\infty} = f\circ F_{\infty}, \text{ for all } f\in \Sym{H}.$ This is a contradiction since every function $f\in \Sym{H}$ that acts non-trivially on $F_{\infty}(H)$ will not satisfy this equation.
\end{proof}
A similar result holds for wreath products $A\wr B$, where it is enough to suppose that the semigroup generated by $\supp{\mu}$ contains two distinct elements with equal projections to $B$ (see the proof of Theorem 3.3 in \cite{Kaimanovich1991} and the discussion after Lemma 1.1 in \cite{Erschler2011}). Below, we provide an example that shows that in $\extsym{H}$ an analogous hypothesis does not suffice to guarantee the non-triviality of the Poisson boundary.

We will consider subgroups of $\extsym{\Z}$ whose projection to $\Z$ are proper subgroups of $\Z$. Among these subgroups, we can find wreath products (Proposition \ref{prop: extsymH contains wreath products whenever H is not co-Hopfian}), which admit finite first moment measures with non-trivial boundary. Hence, in order to find subgroups with trivial boundary we will also need to restrict the possible values for the permutation coordinate of the elements of these subgroups.

\begin{exmp}\label{example: self correcting configuration}
	Let us fix $M\ge 3$ and consider the finite subset $\Sigma_M\leqslant \Sym{\Z}$ formed by all $f\in \Sym{\Z}$ such that $\supp{f}\subseteq [-M,M]$, and $f(x)=x+M$ whenever $-M\le x\le 0$. That is, a permutation $f\in \Sigma_M$ coincides with the identity outside of $[-M,M]$, acts as a translation by $M$ on the interval $[-M,0]$ and maps bijectively the set $[1,M]$ to $[-M,-1]$.
	
	Let $K$ be the subgroup of $\extsym{\Z}$ generated by elements of the form $(f,M+1)$, for $f\in \Sigma_M$. Let us denote by $S_K$ the set formed by the above generators, together with their inverses. Note that for $f\in \Sigma_M$, the inverse of $(f,M+1)$ is given by $(\tilde{f},-(M+1))$, where 
	\begin{equation*}
		\tilde{f}(x)=\begin{cases}
			x+M, &\text{ if }x\in[-M-1,-1],\\
			-M-1+f^{-1}(x+M+1), &\text{ if }x\in[-2M-1,-M-2], \text{ and}\\
			x,&\text{ otherwise}.
		\end{cases}
	\end{equation*}
	
	In particular, we see that the values of $\tilde{f}$ are uniquely determined on the interval $[-M-1,1]$. Furthermore, when multiplying two elements $(f_1,j_1(M+1)),(f_2,j_2(M+1)) \in S_K$, one obtains \begin{equation*}(f_1,j_1(M+1))\cdot (f_2,j_2(M+1)) =(f_3,(j_1+j_2)(M+1)),\end{equation*} 
	where
	\begin{equation*}
		f_3=f_1 \circ \left(j_1(M+1)\cdot f_2\right),
	\end{equation*}
	and the places where $f_3$ is not uniquely determined are a translation of those of $f_2$, which is a set of size $M$. From this, one can see that the elements of $K$ are all of the form $g=(f,j(M+1))$, for $j\in \Z$ and for $f\in \Sym{\Z}$ a permutation whose values are uniquely determined by the value of $j$, except for those of the interval $[(j-1)(M+1)+1,j(M+1)-1]$, which has size $M$. Further, if the element $g$ is a product of $n$ generators in $S_K$, then it holds that $|j|\le n$. Thus, the growth function of $K$ is bounded above by $nM!$ and as a result, the group $K$ has a linear growth function. 
	
	Gromov's Theorem on groups of polynomial growth \cite{Gromov1981} implies that $K$ is virtually nilpotent, and the Bass–Guivarc'h formula \cite{Bass1972,Guivarch1973} shows that any virtually nilpotent group of linear growth is virtually $\Z$. We conclude that $K$ is virtually cyclic, and hence any random walk supported on $K$ has a trivial Poisson boundary. Nonetheless, Lemma \ref{lem: main stabilization lemma} implies that the permutation coordinate will stabilize for a given measure $\mu$ with a finite first moment measure supported on $K$ with a non-centered projection to $\Z$.  Since we already know that the Poisson boundary is trivial, it must hold that almost every trajectory of the random walk stabilizes to the same limit function. Indeed, if the projection of $\mu$ to $\Z$ has positive drift, then the limit function $F_{\infty}$ is almost surely given by $F_{\infty}(x)=x+M$, $x\in \Z$, whereas if the drift is negative, the limit function is $F_{\infty}(x)=x-M$, $x\in \Z$.
	
\end{exmp}

It is natural to draw a parallel with Proposition \ref{prop: non-degenerate measure + stabilization imply non-trivial boundary}. In the above example, the permutation coordinate of every trajectory stabilizes to the same limit function $F$, which is not surjective. In the case of positive drift in the projection to $\Z$, the image of $F$ is  $\Z\backslash [-M,-1]$. We saw that any element $g\in K$ is of the form $g=(f,j(M+1))$, with $j\in \Z$ and where the values of $f$ are uniquely determined, except for those of the interval $[(j-1)(M+1) +1,j(M+1)-1]$. When looking at the action of $g$ on the limit function $F$ we get
\begin{equation*}
	g\cdot F=f\circ (j(M+1)\cdot F),
\end{equation*} 
and we note that the function $j(M+1)\cdot F$ has as its image $\Z\backslash [(j-1)(M+1) +1,j(M+1)-1]$, which is exactly the set where $f$ is uniquely determined. In other words, all the elements of $K$ act trivially on $F$.

We now state and prove the stabilization lemma in its general form, and then prove Lemma \ref{lem: main stabilization lemma}. In addition to transience of the random walk induced on $H$, the second hypothesis below is that $\E\left( |\supp{\sigma_1}| \right)<\infty$. This means that the number of elements on which a randomly chosen permutation $\sigma_1$ acts non-trivially has a finite expectation. 

\begin{lem}\label{lem. stabilization key lemma (Borel-Cantelli)} Let $H$ be a countable group and $\mu$ be a probability measure on $\extsym{H}.$ Suppose that $\mu$ induces a transient random walk on $H$ and that $\E\left( |\supp{\sigma_1}| \right)<\infty$. Then the permutation coordinate of the $\mu$-random walk on $\extsym{H}$ stabilizes.
\end{lem}
\begin{proof}
	Fix an arbitrary element $h\in H$, and for every $n\ge 1$ consider the event $A_n=\left\{ F_{n+1}(h)\neq F_n(h)\right\}.$
	We will prove that $\sum_{n\ge 0} \P(A_n)<+\infty$, so that the Borel-Cantelli Lemma \cite[Lemma VIII.3.1]{Feller1968} will imply that almost surely only finitely many of these events happen. Since $H$ is countable, the above implies that the stabilization of $F_n(h)$ for every $h\in H$ happens with probability $1$.

	Using the group operation and the definition of the action of $H$ on $\Sym{H}$, we see that
	$$
	F_{n+1}(h)=F_n\circ (S_n\cdot \sigma_{n+1}) (h)= F_n \left(S_n\sigma_{n+1}(S^{-1}_n h) \right),
	$$
	so that $A_n\subseteq\{\sigma_{n+1}(S^{-1}_n h)\neq S^{-1}_n h\}.$ With this, we can obtain an upper estimate for the probability of $A_n$. Indeed,
	\begin{align*}
		\P(A_n)&\le\P\left(\sigma_{n+1}(S^{-1}_n h)\neq S^{-1}_n h \right)\\
		&=\sum_{ x\in H }\P\left(\sigma_{n+1}(x^{-1} h)\neq x^{-1} h \right)\P(S_n=x)\\
		&=\sum_{ x\in H }\P\left(x^{-1} h \in \supp{\sigma_{n+1}} \right)\P(S_n=x)\\
		&=\sum_{ x\in H }\sum_{(f,y)\in \extsym{H}}\mathds{1}_{\left\{x^{-1} h \in \supp{f}\right\}}\mu(f,y)\P(S_n=x).
	\end{align*}
	
	Now we are going to sum over all $n$. Note that since we assume that the projection of $\mu$ to $H$ is transient, there exists a constant $C>0$ such that for every $x\in H$, $\sum_{n\ge 1}\P(S_n=x)<C.$
	Indeed, the left-hand side is the expected number of visits to $x$, which is equal to the probability of ever reaching $x$ multiplied by the expected number of visits to $e_H$. The first term is at most $1$ (since it is a probability), and the second one is a finite constant, thanks to our transience hypothesis. 
	
	We can conclude that
	\begin{align*}
		\sum_{n\ge 1} \P(A_n) &\le \sum_{n\ge 1}\sum_{ x\in H }\sum_{(f,y)\in \extsym{H}}\mathds{1}_{\left\{x^{-1} h \in \supp{f}\right\}}\mu(f,y)\P(S_n=x)\\
		&\le C \sum_{(f,y)\in \extsym{H}}\sum_{ x\in H }\mathds{1}_{\left\{x^{-1} h \in \supp{f}\right\}}\mu(f,y)\\
		&= C\sum_{(f,y)\in \extsym{H}}|\supp{f}|\mu(f,y)=C\E(|\supp{\sigma_1}|)<\infty.
	\end{align*}
	which is finite thanks to our hypothesis.
\end{proof}
\begin{proof}[Proof of Lemma \ref{lem: main stabilization lemma}]
	
	When $H$ is finitely generated, so is $\extsym{H}$ and the condition $\E(|\supp{\sigma_1}|)<\infty$ follows from the finite first moment hypothesis. Indeed, recall the definition of $\std$ from Lemma \ref{lem: basic properties of FSym(H)rtimes H}. Each transposition in $\std$ changes the value of exactly two elements of $H$, and hence a geodesic word for $(f,y)$ needs at least as many transpositions as half the size of the support of $f$. In other words, we have the inequality $|\supp{f}|\le 2\|(f,y)\|_{\std}$ for every $(f,y)\in \extsym{H}$. This implies that
	\begin{align*}
		\E(|\supp{\sigma_1}|)&=	\sum_{(f,y)\in \extsym{H}}|\supp{f}|\mu(f,y)\\&\le \sum_{(f,x)\in \extsym{H}}2\|(f,x)\|_{\std}\mu(f,x)<\infty,
	\end{align*}
	and so we can apply Lemma \ref{lem. stabilization key lemma (Borel-Cantelli)}.
\end{proof}

By combining Proposition \ref{prop: non-degenerate measure + stabilization imply non-trivial boundary} with Lemma \ref{lem. stabilization key lemma (Borel-Cantelli)} we obtain the following corollary.

\begin{cor}\label{cor: non degenerate + transient + finite expected support implies nontrivial poisson boundary}
	Let $\mu$ be a non-degenerate probability measure on $\extsym{H}$ that induces a transient random walk on $H$, such that $\E(|\supp{\sigma_1}|)<\infty$. Then the Poisson boundary of $(\extsym{H},\mu)$ is non-trivial.
\end{cor}

\begin{rem}\label{rem: not stabilization of FSym Z3 for nondegenerate rw}
	It may happen that the permutation coordinate of a random walk on $\extsym{H}$ with a transient projection to $H$ does not stabilize. Indeed, whenever $H$ is amenable, the group $\extsym{H}$ is amenable (Lemma \ref{lem: basic properties of FSym(H)rtimes H}) and it was shown by Rosenblatt \cite{Rosenblatt1981} and Kaimanovich and Vershik \cite[Theorem 4.4]{KaimanovcihVershik1983} that every amenable group admits a non-degenerate probability measure with a trivial Poisson boundary. For such $\mu$, Proposition \ref{prop: non-degenerate measure + stabilization imply non-trivial boundary} implies that the permutation coordinate cannot stabilize. Furthermore, if $H$ is not virtually $\Z$ or virtually $\Z^2$, the projected random walk to $H$ will be transient, due to \cite{Varopoulos1986} (see also \cite[Theorem 3.24]{Woess2000}).
\end{rem}

The above remark implies that $\extsym{\Z^d}$, $d\ge 3$, carries a random walk with a non-degenerate symmetric step distribution, with a transient projection to $\Z^d$ and such that the permutation coordinate does not stabilize. On the other hand, the existence of such measures for $d=1$ or $d=2$ is not immediate since the projected random walk to the corresponding base group could be recurrent. In the next proposition, we prove that for $H=\Z$ one can choose these measures so that the projection to $\Z$ is transient, by modifying the proofs of \cite[Theorem 4.4]{KaimanovcihVershik1983} and \cite[Theorem 1.10]{Rosenblatt1981}.

\begin{prop}\label{prop: existence of measure with trivial boundary but transient projection} There exists a non-degenerate symmetric probability measure $\mu$ on $\extsym{\Z}$, that induces a transient random walk on $\Z$ and such that the Poisson boundary of $(\extsym{\Z},\mu)$ is trivial.
\end{prop}
\begin{proof}
	Let us write $G=\extsym{\Z}$ and $\pi:G\to \Z$ the projection map. 
	
	Let $\{K_i\}_{i\ge 1}$ be an increasing sequence of finite subsets of $G$ such that $e\in K_1$ and $G=\bigcup_{i\ge 1}K_i$. Let us consider
	\begin{itemize}
		\item A sequence $\{t_i\}_{i\ge 1}$ of positive numbers such that $\sum_{i\ge 1}t_i=1$,
		
		\item a decreasing sequence $\{\varepsilon_i\}_{i\ge 1}$ of positive numbers such that $\sum_{i\ge 1}\varepsilon_i<+\infty$, and
		
		\item sequences of integers $\{n_i\}_{i\ge 1}$ and $\{p_i\}_{i\ge 1}$ with $n_i,p_i\xrightarrow[i\to \infty]{}+\infty$ such that $$(t_1+t_2+\cdots +t_{i-1})^{n_i}\le \varepsilon_i \text{   and   } (t_1+t_2+\cdots + t_{p_i-1})^i\le \varepsilon_i.$$
	\end{itemize}
	For example, one can choose $t_i=2^{-i}$, $p_i=\lfloor \log(i)\rfloor +1$, $\varepsilon_i=\left(1-2^{- \log(i)}\right)^i$ and 
	$$
	n_i=\left\lceil \frac{i\log\left( 1-2^{-\log(i)}\right) }{\log\left({1-2^{-i+1}} \right)} \right \rceil.
	$$
	Since $G$ is amenable, we can find a sequence of symmetric F\o lner sets $\{A_m\}_{m\ge 1}$ such that, denoting $\alpha_m=\frac{1}{|A_m|}\chi_{A_m}$ the uniform probability measure on $A_m$, we have
	\begin{equation*}\label{eq: invariance uniform measures on Folner sets}
		\|\alpha_{m}-g\alpha_{m}\|\le \varepsilon_m, \text{ for all }g\in  K_m\cup (A_{m-1})^{n_m}.
	\end{equation*}
	These sets can be chosen so that $A_m$ contains  $K_m\cup (A_{m-1})^{n_m}$ and that $\pi(A_m)\subseteq \Z$ is a symmetric interval around $0$, say $\pi(A_m)=[-N_m,N_m]$, such that $\{N_m\}_{m}$ are increasing positive integers, and that whenever $m\ge p_n$ we have
	\begin{equation}\label{eq: symmetric projection to Z is very large}
		\left(\pi_{*}\alpha_m\right)^{* q}\left(\left[-(n-1)N_{m-1},(n-1)N_{m-1}\right] \right)<\varepsilon_n,
	\end{equation}
	for every $1\le q\le n$. In other words, we guarantee that for every $m\ge p_n$, the first $n$ steps of the projected random walk on $\Z$ (whose increments distribute according to $\pi_{*}\alpha_m$) stay out of the set $\left[-(n-1)N_{m-1},(n-1)N_{m-1}\right]$ with high probability. The measure we are looking for is $\mu\coloneqq \sum_{m\ge 1}t_m \alpha_m.$
	
	Indeed, $\mu$ is a symmetric non-degenerate measure on $G$, which has a trivial Poisson boundary, just as in the proof \cite[Theorem 4.3]{KaimanovcihVershik1983}. Denote $\nu=\pi_{*}\mu$ the projection of $\mu$ to $\Z$, and let us prove that $\nu$ induces a transient random walk on $\Z$.
	
	By construction, we have $\nu=\sum_{m\ge 1}t_m \beta_m,$ where $\beta_m=\pi_{*}\alpha_m$. We will prove that for every $n\ge 1$ one has $\nu^{*n}(0)\le 2\varepsilon_n,$ and since $\{\varepsilon_n\}_{n\ge 1}$ is summable, this will imply that the $\nu$-random walk on $\Z$ is transient.

	Note that 
	$$
	\nu^{*n}=\sum_{\mathbf{k}}t_{k_1}\cdots t_{k_n}\beta_{k_1}*\cdots * \beta_{k_n},
	$$
	where $\mathbf{k}=(k_1,\ldots,k_n)$ ranges over all possible multi-indices. We write $\nu^{*n}=\gamma_1+\gamma_2$, where
	$$
	\gamma_1=\sum_{|\mathbf{k}|<p_n}t_{k_1}\cdots t_{k_n}\beta_{k_1}*\cdots * \beta_{k_n}, \text{ for }|\mathbf{k}|=\max_{1\le i\le n}k_i,
	$$
	and $\gamma_2=\nu^{*n}-\gamma_1.$
	
	First, note that our choice of $p_n$ guarantees that
	$$
	\gamma_1(0)\le \sum_{|\mathbf{k}|<p_n}t_{k_1}\cdots t_{k_n}=(t_1+\cdots +t_{p_n-1})^n\le \varepsilon_n.
	$$
	Now let us bound the value of $\gamma_2(0)$. Fix a multi-index $\mathbf{k}$ such that $|\mathbf{k}|\ge p_n$ and consider an index $j$ such that $k_j$ is the largest entry. Since $\Z$ is abelian, the convolution of measures is an abelian operation, and we can write $\beta_{k_1}* \cdots * \beta_{k_n}=\beta_{k_j}^{* q}*\theta,$
	where $q\ge 1$ and $\theta$ is an $(n-q)$-th convolution of $\beta_i$'s with $i<k_j$, so that it satisfies $$\supp{\theta}\subseteq [-(n-1)N_{k_j-1},(n-1)N_{k_j-1}].$$ With this, we get
	\begin{align*}
		\gamma_2(0)&=\sum_{x\in \Z}\beta_{k_j}^{*q}(x)\theta(-x)\\
		&=\sum_{|x|\le (n-1)N_{k_j-1}}\beta_{k_j}^{*q}(x)\theta(-x)\\
		&\le \sum_{|x|\le (n-1)N_{k_j-1}}\beta_{k_j}^{*q}(x)\\
		&=\beta_{k_j}^{*q}\left(\left[-(n-1)N_{k_j-1},(n-1)N_{k_j-1}\right]\right)\\
		&\le \varepsilon_n, 
	\end{align*}
	where we used the fact that $\beta_{k_j}$ satisfies Equation \eqref{eq: symmetric projection to Z is very large}, with $k_j\ge p_n$.
	
	We conclude that $\nu^{* n}(0)=\gamma_1(0)+\gamma_2(0)\le 2\varepsilon_n,$ which finishes the proof.
\end{proof}
If we weaken the hypothesis $\E\left( |\supp{\sigma_1}| \right)<\infty$ from Lemma \ref{lem. stabilization key lemma (Borel-Cantelli)} it is possible that the permutation coordinate never stabilizes. Indeed, with ideas similar to an example of \cite{Kaimanovich1983}, we obtain the following.

\begin{prop} \label{prop: a first moment condition is necessary for stabilization} The group $\extsym{\Z}$ admits probability measures $\mu$ with an infinite first moment and a finite $(1-\varepsilon)$-moment, for every $0<\varepsilon<1$, that induce a transient random walk on $\Z$ and for which the permutation coordinate of the $\mu$-random walk does not stabilize. Such measures can be chosen to satisfy $\E(|\supp{\sigma_1}|)=\infty$ and $\E(|\supp{\sigma_1}|^{1-\varepsilon})<\infty$ for every $0<\varepsilon<1$.
\end{prop}
\begin{proof}	
	For each $n\ge 1$, denote by $r_n:\Z\to \Z$ the permutation
	\begin{equation*}
		r_n(x)=\left\{
		\begin{aligned}
			x+1,& \text{ if }0\le x< n-1,\\
			0,& \text{ if }x=n-1, \text{ and}\\
			x,& \text{ otherwise.}	
		\end{aligned}\right.
	\end{equation*}
	We define the measure $\mu$ on $\extsym{\Z}$ as follows. Let \begin{equation*}\mu((\id,1))=1/8, \ \mu((\id,-1))=3/8,
	\end{equation*} 
	and
	\begin{equation*}
		\mu((r_n,0))=\frac{1}{2n(n+1)}, \ \text{ for }n\ge 1.
	\end{equation*}
	Note that $\sum_{n\ge 1}\frac{1}{n(n+1)}=1$, so that $\mu$ is indeed a probability measure. 
	Also note that $|\supp{r_n}|=n$. From this, the fact that the harmonic series $\sum_{n\ge 1}\frac{1}{n}$ diverges implies that $\E(|\supp{\sigma_1}|)$ is infinite.  Moreover, since $\|(r_n,0)\|_{\std}\ge |\supp{r_n}|$, we also have that $\mu$ has an infinite first moment. On the other hand, for every $\varepsilon>0$ the series $\sum_{n\ge 1} \frac{n^{1-\varepsilon}}{n(n+1)}$ is convergent and thus $\E(|\supp{\sigma_1}|^{1-\varepsilon})$ is finite. The element $r_n$ has word length at most $3n$ (since it can be expressed as the product of at most $n$ transpositions together with $2n$ movements in the $\Z$ coordinate), and hence
	\begin{equation*}
		\sum_{n\ge 1} \frac{\|(r_n,0)\|_{\std}^{1-\varepsilon}}{n(n+1)}\le 	3^{1-\varepsilon}\sum_{n\ge 1} \frac{n^{1-\varepsilon}}{n(n+1)},
	\end{equation*}
	which is finite. Hence, $\mu$ has a finite $(1-\varepsilon)$-moment.
	
	Let us show that the value $F_n(0)$, $n\ge 1$, almost surely changes infinitely often. By definition of the group operation and the $\mu$-random walk, we can write $F_n=F_{n-1}\circ (S_n\cdot \sigma_n)$. Hence, $F_n(0)\neq F_{n-1}(0)$ if and only if $S_n\cdot \sigma_n (0)\neq 0$, which can be rewritten as $\sigma_n(-S_n)\neq -S_n$, by using the definition of the action of $\Z$ on $\Sym{\Z}$ (here we use an additive notation for the group operation on $\Z$).
	
	The induced random walk on $\Z$ is drifted to the negative numbers, and hence almost surely $S_n\xrightarrow[n\to \infty]{}-\infty$. Also, at time $n$ the projection to $\Z$ satisfies $S_n\ge -n$, since the distribution of the increments of the induced random walk on $\Z$ is supported on $\{1,-1\}$. 
	Consider the event $A_n=\{\sigma_n(i)\neq i \text{ for }0\le i<n\}$, and note that
	\begin{equation*}
		\P(A_n)\ge \sum_{k>n}\P(\sigma_n=r_k)=\sum_{k>n}\mu((r_k,0))=\sum_{k>n}\frac{1}{2k(k+1)}=\frac{1}{2(n+1)}.
	\end{equation*}
	
	Since the sequence of events $\{A_n\}_{n\ge 1}$ is independent and the series $\sum_{n} \P(A_n)$ diverges, the Borel-Cantelli Lemma implies that almost surely infinitely many of these events occur. In consequence, the value of the permutation coordinate at $0$ changes infinitely often.
\end{proof}
In Kaimanovich's example mentioned above, the difference between the states of two adjacent lamps does stabilize \cite{Kaimanovich1983}, so that the Poisson boundary is non-trivial. A similar example is described by Lyons and Peres after the proof of Theorem 5.1 in \cite{LyonsPeres2021}. Examples of random walks on $\Z/2\Z\wr \Z^d$, $d\ge 1$, which have a non-trivial Poisson boundary but for which there is no functional defined by a finite set that stabilizes along infinite trajectories are given in \cite[Section 6]{Erschler2011}.
\section{Proof of the main theorem}\label{section: proof of the main theorem}
Let us consider a finitely generated group $H$. Let $\mu$ be a probability measure on $\extsym{H}$ with a finite first moment and a transient projection to $H$. In this case, Lemma \ref{lem: main stabilization lemma} implies that the permutation coordinate $(F_n)_n$ of the $\mu$-random walk stabilizes to a limit injective function $F_{\infty}:H\to H$. As a result, the space $\mathcal{F}(H)\coloneqq \{f:H\to H\mid f\text{ is injective}\}$ has the structure of a measure space $(\mathcal{F}(H),\lambda)$, where $\lambda$ is the \emph{hitting measure} and satisfies for any $A\subseteq \mathcal{F}(H)$ measurable, $\lambda(A)\coloneqq \P(F_{\infty}\in A).$ Alternatively, $\lambda$ is the push-forward of $\P$ through the map $\extsym{H}^{\infty}\to \mathcal{F}(H)$ that associates with every sample path of the $\mu$-random walk on $\extsym{H}$ the associated limit function $F_{\infty}$ of the permutation coordinate. Note that this map is shift-invariant, which implies that the measure $\lambda$ is $\mu$-stationary. 

The space $(\mathcal{F}(H),\lambda)$ is thus a $\mu$-boundary, as described in Subsection \ref{subsection: the Poisson boundary}. That is, it is a quotient of the Poisson boundary. In this section, we prove Theorem \ref{thm. main theorem}, which states that for $H=\Z$, the $\mu$-boundary $(\mathcal{F}(\Z),\lambda)$ actually coincides with the Poisson boundary of the random walk $(\extsym{\Z},\mu)$.

The proof of Theorem \ref{thm. main theorem} uses Kaimanovich's Conditional Entropy Criterion (Theorem \ref{thm: entropy criterion formulation}). The main idea is that conditioned on the limit function to which the permutation coordinate converges, for every $\varepsilon>0$ and any large enough $n$ we can find a finite subset $Q_n\subseteq \extsym{\Z}$ with $|Q_n|<\exp(\varepsilon n)$, and such that $(F_n,S_n)\in Q_n$ with some fixed positive probability.

The fact that $\mu$ has a finite first moment implies that the projection $\mu_{\Z}$ of $\mu$ to $\Z$ also does. Since $\mu_{\Z}$ induces a transient random walk, it holds that the $\mu_{\Z}$-random walk on $\Z$ has non-zero drift $\sum_{x\in \Z}x\mu_{\Z}(x)$. The law of large numbers then allows us to confine the position coordinate $S_n$ within an interval $I_n$ of length $2\varepsilon n$ and to estimate the values of the permutation coordinate outside this interval. However, this is not enough for our purposes, since a rough estimate for the number of values that the permutation coordinate can take inside $I_n$ is $(2\varepsilon n)!$, which leads to sets $Q_n$ that have an exponential size. To overcome this difficulty, we look at the possible values for the displacement of $F_n$. Recall that the displacement of $\sigma\in \extsym{\Z}$ is defined as $\mathrm{Disp}(\sigma)=\sum_{i\in \Z}|\sigma(i)-i|$ (Definition \ref{defn: displacement}). The first moment hypothesis gives us control over the possible values for the displacement associated with the permutation increments that modify the values in the interval $I_n$, which in turn reduces the previously mentioned estimate into a subexponential one.

\subsection{The proof}\label{subsection: the proof}
We first state two preliminary lemmas that follow from the hypothesis of $\mu$ having a finite first moment and the Strong Law of Large Numbers \cite[Section VIII.4]{Feller1968}. Then, we proceed with the proof.

\begin{lem}\label{lem: increments are not that big}
	Consider a probability measure $\mu$ on $\extsym{\Z}$ with a finite first moment. Then there exists a constant $D>0$ such that for every $\varepsilon>0$ and almost every sequence of i.i.d.\ increments $\{(\sigma_k,X_k)\}_{k\ge 1}$ there exists $N\ge 1$ such that for every $n\ge N$ one has
	\begin{enumerate}
		\item $|X_n|\le \varepsilon n$,
		\item $\sigma_n\in \mathrm{Sym}\left(\left[-\varepsilon n,\varepsilon n\right]\right)$, and
		
		\item $\displaystyle \sum_{k=n-\varepsilon n}^n \mathrm{Disp}(\sigma_k)<D \varepsilon n$.
		
	\end{enumerate}
\end{lem}
\begin{proof}
	We will use the Borel-Cantelli Lemma for the first two items. We see that
	\begin{align*}
		\sum_{n\ge 1}\P\left(|X_n|\ge \varepsilon n\right)&=\sum_{n\ge 1}\sum_{(F,x)\in \extsym{\Z}}\mathds{1}_{\left\{|x|\ge \varepsilon n\right\} }\mu(F,x)\\
		&=\sum_{(F,x)\in \extsym{\Z}}\frac{1}{\varepsilon}|x|\mu(F,x)\\
		&\le\frac{1}{\varepsilon}\sum_{(F,x)\in \extsym{\Z}}\|(F,x)\|_{\std}\mu(F,x)<+\infty.
	\end{align*}
	
	Similarly, we have
	
	\begin{align*}
		\sum_{n\ge 1}\P\left(\supp{\sigma_n}\nsubseteq \left[- \varepsilon n, \varepsilon n\right]\right)&=\sum_{n\ge 1}\sum_{(F,x)\in \extsym{\Z}}\mathds{1}_{\left\{\supp{F}\nsubseteq \left[- \varepsilon n,\varepsilon n\right]\right\} }\mu(F,x)\\
		&\le \sum_{n\ge 1}\sum_{(F,x)\in \extsym{\Z}}\mathds{1}_{\left\{ n\le\frac{1}{\varepsilon} \max\{|y|:\  y\in \supp{F}\} \right\} }\mu(F,x)\\
		&\le \frac{1}{\varepsilon}\sum_{(F,x)\in \extsym{\Z}}\|(F,x)\|_{\std}\mu(F,x)<+\infty,
	\end{align*}
	
	Since both sums are finite, the Borel-Cantelli Lemma implies that almost surely these events occur finitely many times, and thus neither of them happens for $n$ sufficiently large.
	
	The third item follows from the Strong Law of Large Numbers, since the sequence of random variables $\{\mathrm{Disp}(\sigma_k)\}_{k\ge 1}$ are i.i.d. of finite first moment. Indeed, using Lemma \ref{lem: word metric estimates}, we have
	$$
	\E(\mathrm{Disp}(\sigma_1))\le 2\E(\|(\sigma_1,X_1)\|_{\std})<+\infty.
	$$
\end{proof}

The next lemma is a direct consequence of the Strong Law of Large Numbers.

\begin{lem}\label{lem: asymptotics Z} Fix a probability measure $\nu$ on $\Z$ with finite first moment and positive drift. Denote by $\{S_n\}_{n\ge 0}$ the associated $\nu$-random walk on $\Z$. Then there exists a constant $C_1>0$ such that
	for every $\varepsilon>0$, almost surely there exists $N\ge 1$ such that for $n\ge N$ we have $$C_1n-\varepsilon n \le S_n\le C_1 n +\varepsilon n.$$	
\end{lem}

Let us now proceed with the proof of the main theorem.

\begin{proof}[Proof of Theorem \ref{thm. main theorem}]
	Since the projection of $\mu$ to $\Z$ is transient and has a finite first moment, its drift $\E(X_1)$ is non-zero, and we lose no generality if we furthermore suppose that it is positive. We will assume this throughout the rest of the proof.
	
	Denote by $F_{\infty}$ the random variable defined as the limit of the permutation component $\{F_n\}_{n\ge 0}$ of the $\mu$-random walk on $\extsym{\Z}$, which is almost surely well-defined thanks to Lemma \ref{lem: main stabilization lemma}. To prove the theorem, it suffices to check the hypotheses of Theorem \ref{thm: entropy criterion formulation} for the boundary of limit functions $F_{\infty}$.

	Fix an arbitrary $\varepsilon>0$. Thanks to Lemmas \ref{lem: increments are not that big} and \ref{lem: asymptotics Z}, there exist constants $C_1,D>0$ and $N\ge 1$ large enough so that for every $n\ge N$ we have that 
	$$C_1 n -\varepsilon n \le S_n \le C_1 n +\varepsilon n,$$ and that the increments of the $\mu$-random walk satisfy $$|X_n|<\varepsilon n,\text{ and } \sigma_n\in \Sym{[-\varepsilon n, \varepsilon n]},$$ together with
	$$\displaystyle \sum_{k=n-\tilde{\varepsilon} n}^n \mathrm{Disp}(\sigma_k)<D\tilde{\varepsilon} n,$$ where $\tilde{\varepsilon}=\frac{4\varepsilon}{C_1+2\varepsilon}$. This choice of $\widetilde{\varepsilon}$ is to simplify the computations below.
	
	The above guarantees that $F_n(y)=y$ for every $y>(C_1+2\varepsilon) n$. Indeed, the maximum value that the projection to $\Z$ could have visited is $C_1+\varepsilon n$, and since the permutation component of the increments has a range of at most $\varepsilon n$, there have been no modifications of the permutation component of the random walk beyond $C_1+\varepsilon n +\varepsilon n=(C_1+2\varepsilon)n.$
	
	Similarly, for every instant after $n$ we know that the projection to $\Z$ will not visit any value smaller than $(C_1-\varepsilon)n$, so that for every $y<(C_1-2\varepsilon)n$ the value $F_n(y)$ is already fixed to its limit $F_{\infty}(y)$. Thus, we know the exact value of $F_n(y)$ for every $y\in \Z$ such that $|y-C_1n|>2\varepsilon n$.
	
	We now estimate the possible number of values that $F_n$ can take in the interval $[(C_1-2\varepsilon)n, (C_1+2\varepsilon)n]$. To do so, we remark that if $k_0$ is the first moment that the permutation component of the random walk is not trivial on this interval, then
	\begin{equation}\label{eq: displacement at time n is at most the sum of displacements of each increment}
		\displaystyle \sum_{i=(C_1-2\varepsilon)n}^{(C_1+2\varepsilon)n} |F_n(i)-i|\le \sum_{k=k_0}^{n}\mathrm{Disp}(\sigma_k).
	\end{equation}
	
	For $n$ large enough, the smallest possible value for $k_0$ is $\frac{C_1-2\varepsilon}{C_1+2\varepsilon}n$. Indeed, the maximum value of $S_k$ is $C_1k+\varepsilon k$ and the support of the permutation component increments allows for an extra $\varepsilon k$, so that we get the inequality 
	\begin{equation*}(C_1+2\varepsilon)k_0\ge (C_1-2\varepsilon)n.
	\end{equation*}
	
	Now note that \begin{equation*}n-\frac{C_1-2\varepsilon}{C_1+2\varepsilon}n= \frac{4\varepsilon}{C_1+2\varepsilon}n,
	\end{equation*}
	
	and recall that we defined $\tilde{\varepsilon}=\frac{4\varepsilon}{C_1+2\varepsilon}$, so that we have
	
	\begin{equation*}\sum_{k=n-\tilde{\varepsilon} n}^n \mathrm{Disp}(\sigma_k)<D\tilde{\varepsilon} n=\frac{4D}{C_1+2\varepsilon}\varepsilon n.
	\end{equation*}
	Denote $D'=\frac{4D}{C_1+2\varepsilon}$.
	Thanks to Equation \eqref{eq: displacement at time n is at most the sum of displacements of each increment}, we can interpret the above as saying that the permutation $F_n$ must assign a non-negative value $d_i\coloneqq |F_n(i)-i|$ to each element $i\in [(C_1-2\varepsilon)n, (C_1+2\varepsilon)n]$ such that $$\displaystyle\sum_{i=(C_1-2\varepsilon)n}^{(C_1+2\varepsilon)n}d_i<D'\varepsilon n.$$
	
	The number of ways to do this is the same as the number of ways of distributing $D'\varepsilon n$ identical balls into $4\varepsilon n+1$ distinguishable boxes, together with a factor of $2^{4\varepsilon n+1}$ which accounts for the fact that for the same value of $d_i$ there are at most two choices of $F_n(i)$ (depending on whether $F_n(i)\ge i$ or $F_n(i)<i$). This gives an upper bound of
	$$
	2^{4\varepsilon n+1} \cdot {(4+D')\varepsilon n \choose D'\varepsilon n }
	$$
	for the possible values of the function $F_n$.
	
	To use Kaimanovich's criterion, we define $Q_{n}\subseteq \extsym{\Z}$ to be the set of elements of the form $(F,x)$ where $(C_1-\varepsilon)n\le x\le (C_1+\varepsilon)n$ and where $F$ is a permutation as described above. The random set $Q_n$ is measurable with respect to $\sigma(F_{\infty})$. As a consequence of the above estimates, we have that $(F_n,S_n)\in Q_{n}$ almost surely, for $n$ large enough, and that 
	$$
	|Q_n|\le (4\varepsilon n+1) 2^{4\varepsilon n+1} \cdot {(4+D')\varepsilon n \choose D'\varepsilon n }.
	$$
	Finally, thanks to Stirling's approximation \cite[Section II.9]{Feller1968}, we have
	\begin{align*}
		\limsup_{n\to \infty}\frac{1}{n}\log |Q_n|&\le 4\varepsilon\log 2 + \limsup_{n\to \infty} \frac{1}{n}\log {(4+D')\varepsilon n \choose D'\varepsilon n }\\
		&\le 4\varepsilon\log 2+4\varepsilon\log\left( \frac{4+D'}{4\varepsilon} \right)+D'\varepsilon\log\left( \frac{4+D'}{D'} \right)\\
		&\le C_2 \varepsilon,
	\end{align*}
	for some constant $C_2>0$. With this, we have checked the hypotheses of Theorem \ref{thm: entropy criterion formulation} and hence finished the proof.
\end{proof}

\section{Boundary triviality for recurrent projections to the base group} \label{section: recurrent projection}

Theorem \ref{thm. main theorem} describes the Poisson boundary of random walks on $\extsym{\Z}$ with a finite first moment and a transient projection to $\Z$. On the other hand, the situation is more delicate when the projection to $\Z$ is recurrent since the group $\Sym{\Z}$ admits measures with a non-trivial Poisson boundary \cite{Kaimanovich1983}. In this section, we prove that the Poisson boundary is trivial for finitary measures on $\extsym{\Z}$ or $\extsym{\Z^2}$ that induce a recurrent random walk on the base group.

We first show the result for $H=\Z$, following similar ideas to \cite[Proposition 6.2]{KaimanovcihVershik1983}.
\begin{prop}\label{prop: recurrent proj to Z has trivial Poisson boundary} Consider $\mu$ a finitely supported probability measure on $\extsym{\Z}$ such that its projection to $\Z$ induces a recurrent random walk. Then the Poisson boundary of $(\extsym{\Z},\mu)$ is trivial.
\end{prop}
\begin{proof}

	Denote by $\nu$ the projection of $\mu$ to $\Z$, which has a zero mean and is finitely supported. Consider the $\mu$-random walk $\{(F_n,S_n)\}_{n\ge 0}$ on $\extsym{\Z}$.
	We first recall that Kolmogorov's Maximal Inequality \cite[Section IX.7]{Feller1968} states that for any $\lambda>0$, we have
	
	\begin{equation*}\P\left( \max_{1\le k\le n} |S_k|\ge \lambda \right)\le n\sigma^2/\lambda^2,
	\end{equation*}
	where $\sigma^2$ is the variance of $\nu$.
	In particular for $\lambda=n^{3/4}$, we get
	
	$$
	\P\left( \max_{1\le k\le n} |S_k|< n^{3/4} \right)\ge 1- \frac{\sigma^2}{n^{1/2}}.
	$$
	
	Thus, with probability at least $1-\sigma^2/n^{1/2}$, the random walk at time $n$ belongs to the set
	$$
	A_n=\left \{ (f,x)\in G\mid |x|\le n^{3/4}, \ \supp{f}\subseteq [-n^{3/4}-M,n^{3/4}+M] \right \},
	$$
	where $M=\max\left \{C\ge 0\mid \text{for every }(f,x)\in \supp{\mu}\text{ and }|y|>C, \ f(y)=y\right \}$ is the size of the largest support of functions $f$ that participate in the support of $\mu$.
	
	The size of the set $A_n$ is subexponential. Indeed, we have
	\begin{align*}
		\frac{1}{n}\log |A_n|\le \frac{1}{n}\log \left( (2n^{3/4}+1)\cdot (2n^{3/4}+1+M)!\right),
	\end{align*}
	which converges to $0$. This follows by applying Stirling's approximation to bound from above the term $\log((2n^{3/4}+1+M)!)$. We can thus apply Theorem \ref{thm: entropy criterion formulation} with the trivial boundary since the sets $A_n$ are deterministic and we have
	$$\P\left((F_n,S_n) \in A_n\right)\ge \P\left(  \max_{1\le k\le n} |S_k|< n^{3/4}  \right)\ge 1-\frac{\sigma^2}{n^{1/2}} >1/2,$$
	for any $n$ large enough.
\end{proof}

In order to prove the analogous result for $H=\Z^2$, we need more detailed calculations. Indeed, in the proof of Proposition \ref{prop: recurrent proj to Z has trivial Poisson boundary} we argued that the projection of the random walk to $\Z$ at time $n$ stayed inside an interval of length of order $n^{3/4}$ with some positive probability. Then a rough estimate for the total number of permutations in this interval gives rise to a subset of $\extsym{\Z}$ whose size is of order $(n^{3/4})!$. This grows subexponentially and hence we can apply Theorem \ref{thm: entropy criterion formulation}. In the case of $\Z^2$, we can similarly argue that with a fixed positive probability, the projection of the random walk to $\Z^2$ at time $n$ has not visited the outside of a ball of radius $C_1 n^{1/2}$, for some constant $C_1>0$. However, since the base group now has quadratic growth, a rough estimate for the number of permutations gives an exponential number of possibilities, and so we cannot apply Theorem \ref{thm: entropy criterion formulation}. We go past these complications by giving more detailed estimates for the range of the projected random walk to $\Z^2$, which is $O(n/\log(n))$, and by looking at the displacement of the permutation coordinate increments.

\begin{prop}\label{prop: recurrent proj to Z2 has trivial Poisson boundary} Consider $\mu$ a finitely supported probability measure on $\extsym{\Z^2}$ such that its projection to $\Z^2$ induces a recurrent random walk. Then the Poisson boundary of $(\extsym{\Z^2},\mu)$ is trivial.
\end{prop}
\begin{proof}
	By using Kolmogorov's maximal inequality on both coordinates of $\Z^2$, we see that for a large enough constant $C_1>0$ we have
	\begin{equation*}
		\P\left( \max_{1\le k\le n}\|S_k\|\le C_1 n^{1/2}  \right)>1/2
	\end{equation*}
	for all $n\ge 1$, where $\|\cdot \|$ denotes the $L_1$ norm on $\Z^2$. That is, $\|(x_1,x_2)\|=|x_1|+|x_2|$, for $(x_1,x_2)\in \Z^2$.
	
	We need to estimate the $n$-th instant $(F_n,S_n)$ of the random walk. Let us denote by $R_n\coloneqq \left|\{ S_0,S_1,\ldots,S_n\} \right| $
	the range of the induced random walk on $\Z^2$. In other words, $R_n$ is the number of distinct elements of $\Z^2$ visited up to time $n$. Since we assume that $\mu$ has finite support, the following limit
	\begin{equation*}\label{eq: range recurrent random walk on Z2}
		\lim_{n\to \infty}\frac{R_n}{n/\log(n)}
	\end{equation*}
	exists almost surely and is positive. This was first proven for the simple random walk on $\Z^2$ in \cite{DvoretzkyErdos1951}. The same arguments in that paper prove that for any finitely supported $\mu$ such that the induced random walk on $\Z^2$ is recurrent, one has $\E(R_n)=C \frac{n}{\log(n)}+o\left(\frac{n}{\log(n)}\right)$ for a constant $C>0$. In addition, it is proven in \cite{JainPruitt1972} that for $\mu$ as above, the strong law of large numbers $\lim_{n\to \infty}\frac{R_n}{E(R_n)}=1$ holds almost surely. The combination of both results implies the statement above. In particular for our proof, the above implies that there exists a constant $C_2>0$ such that almost surely for every large enough $n$ we have $R_n<C_2\frac{n}{log(n)}$.

	With this, the $n$-th instant of the random walk is determined by choosing at most $C_2\frac{n}{\log(n)}$ elements of $\Z^2$ to visit inside the ball of radius $C_1\sqrt{n}$, then choosing one of these elements to be the final position $S_n$ of the random walk on $\Z^2$, and finally choosing the values for the permutation coordinate $F_n$. The first two terms above give a factor of
	\begin{equation*}
		{ (C_1\sqrt{n}+1)^2\choose C_2\frac{n}{\log(n)}}\cdot C_2\frac{n}{\log(n)},
	\end{equation*}
	which grows subexponentially. Indeed, we note that both $\frac{1}{n}\log (n/\log(n))=\frac{1}{n}\log(n)- \frac{1}{n}\log \log n$, and
	
	\begin{equation*}
		\frac{1}{n}\log { (C_1\sqrt{n}+1)^2\choose C_2\frac{n}{\log(n)}}
	\end{equation*}
	
	converge to $0$, as it follows from Stirling's approximation.
	
	In order to estimate the number of possible values for the permutation coordinate $F_n$, we look at the displacement function. Note that we have $\sum_{x\in \Z^2}\|x-F_n(x)\| \le C_3 n,$ for some constant $C_3>0$. Indeed, since $\mu$ has finite support, each increment will map each element to its image in a uniformly bounded neighborhood, and so the total displacement at every step is bounded by a fixed constant.
	
	With the above, we can associate to the permutation $F_n(x)$ values $d_x\ge 0$, for each $x$ in the support of size $C_2\frac{n}{\log(n)}$, whose sum must be at most $C_3n$. The total number of ways of assigning these numbers to a fixed support is given by
	$$
	{C_3n + C_2\frac{n}{\log(n)} -1 \choose C_3 n}.
	$$
	
	Suppose that we have fixed a support of size $C_2\frac{n}{\log(n)}$ as well as the displacements of each element $\{d_x\}_x$. We remark that for every element $x$, its image $F_n(x)$ can be any element that satisfies $\|x-F_n(x)\|=d_x$. For a fixed value of $d_x$ there are $2d_x+1$ such elements, and hence the total number of permutations for this fixed support and displacement is bounded above by
	$\prod_{i=1}^{C_2\frac{n}{\log(n)}}(2d_i+1).$ Since we have the constraint $\sum_{i=1}^{C_2\frac{n}{\log(n)}}d_i=C_3 n$, the above product is maximized when all the values $2d_i+1$ are equal, and hence we have the upper bound
	\begin{equation*}
		\left(\frac{2C_3n+C_2\frac{n}{\log(n)}}{C_2\frac{n}{\log(n)}}\right)^{C_2\frac{n}{\log(n)}}.
	\end{equation*}
	This also grows subexponentially.

	In conclusion, our analysis shows that with a probability of at least $1/2$, the $n$-th instant $(F_n,S_n)$ of the random walk on $\extsym{\Z^2}$ belongs to a set $A_n$ of size bounded above by
	{ \begin{equation*}
			{ (C_1\sqrt{n}+1)^2\choose \frac{C_2n}{\log(n)}}\cdot\frac{ C_2n}{\log(n)}\cdot  {C_3n +\frac{C_2n}{\log(n)} -1 \choose C_3 n}\cdot  	\left(\frac{2C_3n+\frac{C_2n}{\log(n)}}{\frac{C_2n}{\log(n)}}\right)^{\frac{C_2n}{\log(n)}},
	\end{equation*}}
	which grows subexponentially. We can thus apply Theorem \ref{thm: entropy criterion formulation} to conclude that the Poisson boundary is trivial.
\end{proof}

\bibliographystyle{alpha}
\bibliography{biblio}
\end{document}